\documentclass[a4paper,11pt,oneside]{amsart}
\pdfoutput=1

\usepackage[utf8]{inputenc}
\usepackage{bm}
\usepackage{mathtools,amssymb}
\usepackage{esint}
\usepackage{hyperref}
\hypersetup{
    colorlinks=true,
    linkcolor=black,
    filecolor=black,      
    urlcolor=cyan,
    citecolor=black
    }
\usepackage{tikz}
\usepackage{dsfont}
\usepackage{relsize}
\usepackage{url}
\usepackage{xcolor}
\usepackage{graphicx}
\usepackage{mathrsfs}
\usepackage[shortlabels]{enumitem}
\usepackage{lineno}
\usepackage{amsmath}
\usepackage{enumitem}
\usepackage{amsthm} 
\usepackage{verbatim}
\usepackage{dsfont}

\allowdisplaybreaks

\mathtoolsset{showonlyrefs}

\graphicspath{{images/}}

\newtheorem{theorem}{Theorem}[section]
\newtheorem{lemma}[theorem]{Lemma}
\newtheorem{definition}[theorem]{Definition}

\newtheorem{proposition}[theorem]{Proposition}

\newtheorem{remark}[theorem]{Remark}

\title[Inverse fractional conductivity problem]{Stability estimates for the inverse fractional conductivity problem}

\author{Giovanni Covi}
\address{Institut fur Angewandte Mathematik, Ruprecht-Karls-Universit\"at Heidelberg, Im Neuenheimer Feld 205, 69120 Heidelberg, Germany}
\email{giovanni.covi@uni-heidelberg.de}
\author{Jesse Railo}
\address{Department of Pure Mathematics and Mathematical Statistics, University of
Cambridge, Cambridge CB3 0WB, UK}
\email{jr891@cam.ac.uk}
\author{Teemu Tyni}
\address{Department of Mathematics, University of Toronto, Canada}
\email{teemu.tyni@utoronto.ca}
\author{Philipp Zimmermann}
\address{Department of Mathematics, ETH Zurich, Z\"urich, Switzerland}
\email{philipp.zimmermann@math.ethz.ch}
\date{\today}

\newcommand{\R}{{\mathbb R}}

\newcommand{\N}{{\mathbb N}}

\newcommand{\schwartz}{\mathscr{S}}

\newcommand{\tempered}{\mathscr{S}^{\prime}}

\newcommand{\fourier}{\mathcal{F}}
\newcommand{\ifourier}{\mathcal{F}^{-1}}
\newcommand{\vev}[1]{\left\langle#1\right\rangle}

\newcommand{\cdistr}{\mathscr{E}'}
\newcommand{\distr}{\mathscr{D}^{\prime}}
\newcommand{\norm}[1]{\lVert #1 \rVert}
\newcommand{\abs}[1]{\left\lvert #1 \right\rvert}
\newcommand{\ip}[2]{\left\langle #1,#2 \right\rangle}
\DeclareMathOperator{\Div}{div} 
\DeclareMathOperator{\supp}{supp} 
\DeclareMathOperator{\dist}{dist} 

\begin{document}

\maketitle
\begin{abstract}
We study the stability of an inverse problem for the fractional conductivity equation on bounded smooth domains. We obtain a logarithmic stability estimate for the inverse problem under suitable a priori bounds on the globally defined conductivities. The argument has three main ingredients: 1. the logarithmic stability of the related inverse problem for the fractional Schrödinger equation by Rüland and Salo; 2. the Lipschitz stability of the exterior determination problem; 3. utilizing and identifying nonlocal analogies of Alessandrini's work on the stability of the classical Calderón problem. The main contribution of the article is the resolution of the technical difficulties related to the last mentioned step. Furthermore, we show the optimality of the logarithmic stability estimates, following the earlier works by Mandache on the instability of the inverse conductivity problem, and by Rüland and Salo on the analogous problem for the fractional Schrödinger equation.
\end{abstract}


\section{Introduction}

Stability estimates for inverse problems give important information on theoretical limitations of different imaging techniques appearing in various medical, engineering, and scientific applications. They are also useful for development of numerical methods. A common feature of many inverse problems is that they are ill-posed, which means that small measurement errors may lead to large errors in the reconstructed images. One of the most popular model problems is the inverse conductivity problem, known as the Calderón problem \cite{C80}, where one aims to recover the conductivity $\gamma$ from the voltage/current measurements on the boundary $\partial\Omega$ of an object $\Omega$. In mathematical terms, one defines the data as a Dirichlet-to-Neumann (DN) map $\Lambda_\gamma\colon f \mapsto \gamma \partial_\nu u_f|_{\partial \Omega}$, where $\partial_\nu$ is the outer boundary normal derivative, the electric potential $u_f$ is the unique solution of the boundary value problem
\begin{equation}
\label{eq: cond eq}
    \begin{split}
            \Div(\gamma\nabla u)&= 0\quad\text{in}\quad\Omega,\\
            u&= f\quad\text{on}\quad \partial\Omega,
        \end{split}
\end{equation}
and the voltage $f$ is the given Dirichlet boundary condition. The Calderón problem asks to recover $\gamma$ from the knowledge of $\Lambda_\gamma$, which corresponds to knowing the outer normal fluxes (i.e. boundary currents) generated by imposing different boundary voltages $f$.

The Calderón problem serves both as a mathematical model for electrical impedance tomography \cite{GU30years}, and more generally as a prototypical model for inverse problems. In fact, methods and techniques originally developed for the classical Calderón problem have applications in a wide range of other inverse problems, among which the anisotropic Calderón problem \cite{AstalaPaivarintaLassasAnisotropic,FKS-Anisotropic}, hyperbolic problems \cite{RakeshSymesHyperbolic,SunHyperbolic} and inverse problems related to the theory of elasticity \cite{NakamuraUhlmannElasticity}. The work of Sylvester and Uhlmann proved a fundamental uniqueness theorem for the classical Calderón problem in dimension $n \geq 3$, using a reduction to an analogous problem for the Schrödinger equation and constructing the so called complex geometrical optics (or CGO) solutions~\cite{SU87}. Nachman established a reconstruction method \cite{Na88}, and Astala--Päivärinta showed a fundamental uniqueness result when $n=2$ using methods from complex analysis and a reduction to the Beltrami equation \cite{AP06}. We recall the stability theorem of Alessandrini \cite{Alessandrini-stability}, which is an important motivation for our present work:
\begin{theorem}[{Alessandrini \cite[Theorem~1]{Alessandrini-stability}}]
Let $\Omega$ be a bounded domain in $\R^n$, $n\geq 3$, with $C^{\infty}$ boundary $\partial\Omega$. Given $s$ and $E$, $s>n/2$, $E>0$, let $\gamma_1,\gamma_2$ be any two functions in $H^{s+2}(\Omega)$ satisfying the following conditions
\[
    \begin{split}
        &E^{-1}\leq \gamma_{\ell}(x),\quad\text{for every }x\text{ in }\Omega,\,\ell=1,2.\\
        &\|\gamma_{\ell}\|_{H^{s+2}(\Omega)}\leq E,\quad\ell=1,2.
    \end{split}
\]
The following estimate holds~\footnote{Given a bounded linear mapping $A\colon X \to Y$ between two Banach spaces, we denote its operator norm by $\norm{A}_{X \to Y}$.}\label{thm: alessandrini-stability}
\[
\|\gamma_1 - \gamma_2\|_{L^\infty(\Omega)} \leq C_{E}\, \omega(\|\Lambda_{\gamma_1} - \Lambda_{\gamma_2}\|_{H^{1/2}(\partial\Omega)\to H^{-1/2}(\partial\Omega)}),
\]
where the function $\omega$ is such that
\[
\omega(t) \leq |\log t|^{-\delta},\quad\text{for every }t,\, 0<t<1/e,
\]
and $\delta$, $0<\delta<1$, depends only on $n$ and $s$.
\end{theorem}

Mandache showed that the logarithmic stability estimates are optimal up to the constants $C,\delta$ \cite{MandacheInstability}. The works of Alessandrini and Mandache therefore show that the classical Calderón problem is ill-posed and furthermore accurately characterize this phenomenon. Mandache's work was recently systematically studied and extended by Koch, Rüland and Salo \cite{KochRulandSaloInstability} to many different settings. For the other recent works on the stability of the classical Calderón problem, we point to the following works \cite{CaroDSFerreiraRuiz, CaroSaloAnisotropic}, where stability under partial data is obtained, and  stability for recovery of anisotropic conductivies is considered. Under certain \emph{a priori} assumptions, such as piecewise constant conductivities, the stronger result of Lipschitz stability holds \cite{AlessandriniVessellaLipschitz}. Lipschitz stability is also possible with a finite number of measurements \cite{AlbertiSantacesariaFinite}. In a different direction, we mention \cite{AbrahamNickl} for an application of stability to the statistical Calderón problem.

In the present work, we study the stability properties of an inverse problem for a nonlocal analogue of the classical Calderón problem. There has been growing interest towards establishing the theory of inverse problems for elliptic nonlocal variable coefficient operators. Other recent studies include inverse problems for the fractional powers of elliptic second order operators \cite{ghosh2021calderon} and inverse problems for source-to-solutions maps related to fractional geometric operators on manifolds \cite{feizmohammadiEtAl2021fractional,QuanUhlmann}. We note that the exterior value inverse problems considered in \cite{ghosh2021calderon} for the operators $(\Div(\gamma \nabla))^s$, $0 < s < 1$, generated by the heat semigroups, give another possibility to define a nonlocal conductivity equation which is presumably different from the equation we study here.

Let $s\in(0,1)$ and consider the exterior value problem for the fractional conductivity equation
\begin{equation}\begin{split}\label{nonlocal-conductivity}
    \mbox{div}_s (\Theta_\gamma\nabla^su)&=0 \quad \mbox{ in } \Omega, \\ u&=f \quad \mbox{ in } \Omega_e ,
\end{split}\end{equation}
where $\Omega_e\vcentcolon=\R^n\setminus\overline{\Omega}$ is the exterior of the domain $\Omega$ and $\Theta_{\gamma}\colon \R^{2n}\to \R^{n\times n}$ is the matrix defined as $\Theta_{\gamma}(x,y)\vcentcolon =\gamma^{1/2}(x)\gamma^{1/2}(y)\mathbf{1}_{n\times n}$. We say $u\in H^s(\R^n)$ is a (weak) solution of \eqref{nonlocal-conductivity} if $u-f \in \widetilde{H}^s(\Omega)$ and
\begin{equation}\label{eq:generalNonlocalOperators}
    B_\gamma(u,\phi):=\frac{C_{n,s}}{2}\int_{\R^{2n}} \frac{\gamma^{1/2}(x)\gamma^{1/2}(y)}{\abs{x-y}^{n+2s}} (u(x)-u(y))(\phi(x)-\phi(y))\,dxdy =0
\end{equation} 
holds for all $\phi \in C_c^\infty(\Omega)$. For all $f\in X\vcentcolon = H^s(\R^n)/\widetilde{H}^s(\Omega)$ in the abstract trace space there is a unique weak solutions $u_f\in H^s(\R^n)$ of the fractional conductivity equation \eqref{nonlocal-conductivity}. The fractional conductivity operator converges in the sense of distributions to the classical conductivity operator when applied to sufficiently regular functions when $s \uparrow 1$ \cite[Lemma 4.2]{covi2019inverse-frac-cond}.

The exterior DN map $\Lambda_{\gamma}\colon X\to X^*$ is defined by 
        \[
        \begin{split}
            \langle \Lambda_{\gamma}f,g\rangle \vcentcolon =B_{\gamma}(u_f,g).
        \end{split}
        \]
The inverse problem for the fractional conductivity equation asks to recover the conductivity $\gamma$ from $\Lambda_{\gamma}$, which maps as $\Lambda_\gamma \colon H^s(\Omega_e)\rightarrow H^{-s}_{\overline \Omega_e}(\R^n)$ in the case of Lipschitz domains. We define $m_\gamma \vcentcolon = \gamma^{1/2}-1$ and call it the \emph{background deviation} of $\gamma$. Let $\Omega \subset \R^n$ be bounded in one direction and $n \geq 1$. (We suppose additionally that $0 < s < 1/2$ when $n=1$.) The uniqueness properties of this inverse problem are studied extensively in the recent literature and the following list summarizes these advances:
\begin{itemize}
    \item \emph{Global uniqueness.} If $W \subset \Omega_e$ is an open nonempty set such that $\gamma_i|_W$ are continuous a.e., and $m_i \in H^{2s,\frac{n}{2s}}(\R^n) \cap H^s(\R^n)$, $j=1,2$, then $\gamma_1=\gamma_2$ if and only if $\Lambda_{\gamma_1}f|_W=\Lambda_{\gamma_2}f|_W$ for all $f \in C_c^\infty(W)$ \cite{RGZ2022GlobalUniqueness}. This result generalizes and expands the scope of the earlier works \cite{covi2019inverse-frac-cond, RZ2022unboundedFracCald}, which solved the inverse problem in certain special cases by means of the fractional Liouville transformation. This is a technique used to reduce the fractional conductivity equation to the fractional Schrödinger equation introduced in \cite{GSU20}, which is in turn better understood.
    \item \emph{Low regularity uniqueness.} If $W \subset \Omega_e$ is an open nonempty set such that $\gamma_i|_W$ are continuous a.e., and $m_i \in H^{s,n/s}(\R^n)$, then $\gamma_1=\gamma_2$ if and only if $\Lambda_{\gamma_1}f|_W=\Lambda_{\gamma_2}f|_W$ for all $f \in C_c^\infty(W)$ \cite{RZ2022LowReg}. This uses a general UCP result for the fractional Laplacians in \cite{KRZ2022Biharm}. 
    \item \emph{Counterexamples for disjoint measurement sets.} For any nonempty open disjoint sets $W_1,W_2\subset\Omega_e$ with $\dist(W_1 \cup W_2, \Omega) > 0$ there exist two different conductivities $\gamma_1,\gamma_2\in L^{\infty}(\R^n) \cap C^\infty(\R^n)$ such that $\gamma_1(x),\gamma_2(x)\geq \gamma_0>0$, $m_1,m_2\in H^{s,n/s}(\R^n) \cap H^{s}(\R^n)$, and $\left.\Lambda_{\gamma_1}f\right|_{W_2}=\left.\Lambda_{\gamma_2}f\right|_{W_2}$ for all $f\in C_c^{\infty}(W_1)$ \cite{RZ2022LowReg}. See the original work on the construction of counterexamples with $H^{2s,\frac{n}{2s}}(\R^n)$ regularity assumptions in \cite{RZ2022counterexamples}, with some limitations in the cases of unbounded domains when $n=2,3$.
\end{itemize}

In this article, we obtain a quantitative stability estimate for the global inverse fractional conductivity problem on \emph{bounded smooth domains} with \emph{full data}. This is based on one of the possible global uniqueness proofs presented in \cite{RGZ2022GlobalUniqueness,RZ2022LowReg}. There remain some nontrivial challenges in order to obtain a quantitative version of the partial data uniqueness results in \cite{RGZ2022GlobalUniqueness,RZ2022LowReg}, as well as to remove the regularity/boundedness assumptions of the domain even for the full data case.

We will next recall two earlier stability results related to the fractional Calderón problems. The first one considers the stable recovery of $\gamma$ in the exterior, based on \cite[Proposition~1.4]{RGZ2022GlobalUniqueness}. The second one considers the stability of the analogous inverse problem for the fractional Schrödinger equation $(-\Delta)^s + q$ due to Rüland and Salo \cite[Theorem~1.2]{RS-fractional-calderon-low-regularity-stability}. The uniqueness properties of the Calderón problem for large classes of fractional Schrödinger type equations have been extensively studied starting from the seminal work of \cite{GSU20}. These include perturbations to the fractional powers of elliptic operators \cite{GLX-calderon-nonlocal-elliptic-operators}, first order perturbations \cite{CLR18-frac-cald-drift}, nonlinear perturbations \cite{LL-fractional-semilinear-problems}, higher order equations with local perturbations \cite{CMR20,CMRU22}, quasilocal perturbations \cite{C21}, and general theory for nonlocal elliptic equations \cite{RS-fractional-calderon-low-regularity-stability,RZ2022unboundedFracCald}.

In particular, the following results are needed in our proofs:
\begin{theorem}[{\cite[Remark~3.3]{RZ2022LowReg}}]
\label{thm: exterior stability}
    Let $\Omega\subset \R^n$ be a domain bounded in one direction and $0<s<1$. Assume that $\gamma_1,\gamma_2\in L^{\infty}(\R^n)$ satisfy $\gamma_1(x),\gamma_2(x)\geq \gamma_0>0$, and are continuous a.e. in $\Omega_e$. There exists a constant $C>0$ depending only on $s$ such that~\footnote{Here and in the rest of paper we use the notation $\norm{A}_*:=\|A\|_{H^s(\Omega_e)\to (H^s(\Omega_e))^*}.$}
    \begin{equation}
    \label{eq: exterior stability}
        \|\gamma_1-\gamma_2\|_{L^{\infty}(\Omega_e)}\leq C\|\Lambda_{\gamma_1}-\Lambda_{\gamma_2}\|_*.
    \end{equation}
\end{theorem}

Given a Sobolev multiplier $q \in M(H^s \to H^{-s})$ (cf.~\cite{RS-fractional-calderon-low-regularity-stability,CMRU22}), we define the following bilinear form
    \[
       B_q(u,v)\vcentcolon =\int_{\R^n}(-\Delta)^{s/2}u\,(-\Delta)^{s/2}v\,dx+\langle qu,v\rangle, \quad u,v \in H^s(\R^n),
     \]
related to the fractional Schrödinger operator $(-\Delta)^s + q$.

\begin{theorem}[{\cite[Theorem~1.2]{RS-fractional-calderon-low-regularity-stability}}]\label{thm: ruland salo proposition}
Let $\Omega\subset\R^n$ be a bounded smooth domain, $0<s<1$, and $W_1,W_2 \subset \Omega$ be nonemtpy open sets. Assume that for some $\delta, M>0$ the potentials $q_1,q_2\in H^{\delta,\frac{n}{2s}}(\R^n)$ have the bounds
\[
\| q_j \|_{H^{\delta,\frac{n}{2s}}(\Omega)}\leq M,\quad j=1,2.
\]
Suppose also that zero is not a Dirichlet eigenvalue for the exterior value problem 
\begin{equation}\label{eq: fractional schrodinger}
    (-\Delta)^s u + q_j u = 0,\quad\text{in }\Omega
\end{equation}
with $u|_{\Omega_e}=0$, for $j=1,2$. Then one has~\footnote{Note $\norm{A}_{\widetilde{H}^s(W_1) \to (\widetilde{H}^{s}(W_2))^*} =\sup\{\,\abs{\langle Au_1,u_2 \rangle} \,;\, \norm{u_j}_{H^s(\R^n)}=1,u_j \in C_c^\infty(W_j)\,\}$.}
\[
\| q_1 - q_2 \|_{L^{\frac{n}{2s}}(\Omega)} \leq \omega(\| \Lambda_{q_1} - \Lambda_{q_2} \|_{\widetilde{H}^s(W_1) \to (\widetilde{H}^s(W_2))^*}),
\]
where $\Lambda_{q_j}\colon X\to X^*$ with $\vev{\Lambda_{q_j}f,g}:=B_{q_j}(u_f,g)$ is the DN map related to the exterior value problem for equation \eqref{eq: fractional schrodinger}, and $\omega$ is a modulus of continuity satisfying
\[
\omega(x) \leq C|\log x|^{-\sigma}, \quad 0< x\leq 1
\]
for some $C$ and $\sigma$ depending only on $\Omega$, $n$, $s$, $W_1$, $W_2$, $\delta$ and $M$.
\end{theorem}

\begin{lemma}[{Liouville reduction, \cite[Lemma 3.9]{RZ2022LowReg}}]
\label{lemma: Liouville reduction}
    Let $0<s<\min(1,n/2)$. Assume that $\gamma\in L^{\infty}(\R^n)$ with conductivity matrix $\Theta_{\gamma}$ and background deviation $m$ satisfies $\gamma(x)\geq \gamma_0>0$ and $m\in H^{s,n/s}(\R^n)$. Let $q_\gamma:=\frac{(-\Delta)^sm}{\gamma^{1/2}}$. Then there holds
    \begin{equation}
    \label{eq: Liouville reduction}
    \begin{split}
        \langle\Theta_{\gamma}\nabla^su,\nabla^s\phi\rangle_{L^2(\R^{2n})}=&\,\langle (-\Delta)^{s/2}(\gamma^{1/2}u),(-\Delta)^{s/2}(\gamma^{1/2}\phi))\rangle_{L^2(\R^n)}\\
        &+\langle q_\gamma(\gamma^{1/2}u),(\gamma^{1/2}\phi)\rangle
    \end{split}
    \end{equation}
    for all $u,\phi\in H^s(\R^n)$.
\end{lemma}
\noindent In the sequel, we will call $q_\gamma$ above  a potential.

\subsection{Main results}

We next state our main result, whose proof is based on a reduction to Theorems~\ref{thm: exterior stability} and \ref{thm: ruland salo proposition}.
\begin{theorem}\label{thm: stability estimate}
Let $0<s<\min(1,n/2)$, $\epsilon>0$ and assume that $\Omega\subset\R^n$ is a smooth bounded domain. Suppose that the the conductivities $\gamma_1,\gamma_2\in L^{\infty}(\R^n)$ with background deviations $m_1,m_2$ fulfill the following conditions:
    \begin{enumerate}[(i)]
        \item\label{item: main assumption 1} $\gamma_0\leq \gamma_1(x),\gamma_2(x)\leq \gamma_0^{-1}$ for some $0<\gamma_0<1$,
        \item\label{item: main assumption 2} $m_1,m_2\in  H^{4s+2\epsilon,\frac{n}{2s}}(\R^n)$
        and there exists $C_1>0$ such that
        \begin{equation}
        \label{eq: main bound 1}
            \|m_i\|_{H^{4s+2\epsilon,\frac{n}{2s}}(\R^n)}\leq C_1
        \end{equation}
        for $i=1,2$,
        \item\label{item: main assumption 3} $m_1-m_2\in H^s(\R^n)$ and there exists $C_2>0$
        \begin{equation}
        \label{eq: difference bound.2}
            \|(-\Delta)^sm_i\|_{L^1(\Omega_e)}\leq C_2
        \end{equation}
        for $i=1,2$.
    \end{enumerate} 
    If $\theta_0\in (\max(1/2,2s/n),1)$ and there holds $\| \Lambda_{\gamma_1}-\Lambda_{\gamma_2}\|_*\leq 3^{-1/\delta}$ for some $0<\delta<\frac{1-\theta_0}{2}$, then we have
    \begin{equation}
    \label{eq: stability estimate}
       \| \gamma_1^{1/2} - \gamma_2^{1/2} \|_{L^q(\Omega)} \leq \omega(\|\Lambda_{\gamma_1} - \Lambda_{\gamma_2}\|_*)
    \end{equation}
    for all $1\leq q \leq \frac{2n}{n-2s}$, where $\omega(x)$ is a logarithmic modulus of continuity satisfying
    \[
     \omega(x) \leq C|\log x|^{-\sigma},\quad\text{for}\quad 0 < x\leq 1,
    \]
    for some constants $\sigma,C>0$ depending only on $s,\epsilon,n,\Omega, C_1,C_2,\theta_0$ and $\gamma_0$.
\end{theorem}

\begin{remark} We make several comments about Theorem \ref{thm: stability estimate} and its assumptions to clarify some interesting points:
\begin{enumerate}[(i)]
    \item Theorems \ref{thm: exterior stability} and \ref{thm: stability estimate} together imply that for any compact set $K \subset \R^n$ there holds
    \begin{equation}
       \| \gamma_1^{1/2} - \gamma_2^{1/2} \|_{L^q(K)} \leq \omega(\|\Lambda_{\gamma_1} - \Lambda_{\gamma_2}\|_*)
    \end{equation}
    where $\omega$ is a logarithmic modulus of continuity with a constant $C$ additionally depending on $K$. In general, one has $L^\infty$ control in $\Omega_e$ and $L^q$ control in $\Omega$.
    \item The $L^1$ assumption \eqref{eq: difference bound.2} is required due to the noncompact, global, setting of the problem, and $L^\infty$ stability in the exterior. The stability estimate in the exterior forces us to impose the additional condition that the related potentials $q_i \vcentcolon = \frac{(-\Delta)^sm_i}{\gamma_i^{1/2}} \in (L^\infty(\Omega_e))^*$ with a priori bounds in their norms, and hence \eqref{eq: difference bound.2} is a natural assumption.
    \item We assume that the domain has smooth boundary due to Theorem \ref{thm: ruland salo proposition}, as one has to impose the smoothness assumption in order to use the Vishik--Eskin estimates. In light of \cite[Remark 7.1]{RS-fractional-calderon-low-regularity-stability}, the rest of their proof could be formulated under weaker regularity assumptions, as well as the one of Theorem \ref{thm: stability estimate} (see Section \ref{sec: partial data reduction}). This leaves the interesting open question of whether it is possible to obtain the stability results, i.e. Theorems \ref{thm: ruland salo proposition} and \ref{thm: stability estimate}, for less regular domains.
    \item We impose the assumption \eqref{eq: main bound 1} so that the related potentials satisfy $q_i \in H^{\delta,\frac{n}{2s}}(\Omega)$ for some $\delta > 0$, and Theorem \ref{thm: ruland salo proposition} applies. The assumption \eqref{eq: main bound 1} also implies that $\gamma_1,\gamma_2$ are continuous, so that Theorem \ref{thm: exterior stability} is known to apply. This assumption is much stronger than the ones required for the global uniqueness theorems \cite{RGZ2022GlobalUniqueness,RZ2022LowReg}.
    \item By formally taking $s=1$ in Theorem \ref{thm: stability estimate} and comparing with Theorem \ref{thm: alessandrini-stability}, one sees that the latter has slightly sharper differentiability assumptions when $n=3$, and the reverse is true for $n\geq 5$. In dimension $n=4$, the assumptions of the two theorems are comparable.
\end{enumerate}
\end{remark}

We complement our stability estimate with the following statement on exponential instability under partial data.
\begin{theorem}\label{thm: instability}
Let $B_1\subset \R^n$ be the unit ball. For any $\ell\in\R_+\setminus\N$ such that $\ell-2s\in\R_+\setminus\N$ as well there exists a constant $\beta>0$ such that for all sufficiently small $\epsilon>0$ there are conductivities $\gamma_1,\gamma_2 \in C^\ell(\R^n)$ such that
$$ \|\Lambda_{\gamma_1}-\Lambda_{\gamma_2}\|_{H^s(B_3\setminus\overline B_2)\rightarrow (H^s(B_3\setminus\overline B_2))^*}  \leq \exp \left(-\epsilon^{-\frac{n}{(2n+3)\ell}}\right), $$
$$ \|\gamma_1-\gamma_2\|_{L^\infty(B_1)}\geq \epsilon, $$
$$ \|\gamma_i\|_{C^\ell(B_1)}\leq \beta, \quad 1\leq\gamma_i\leq 2, \qquad i=1,2. $$
\end{theorem}

\begin{remark} In the above theorem, we let $B_r$, $r>0$, be the ball of radius $r$ centered at the origin. For the sake of simplicity, we restrict our analysis of instability to a very symmetric geometrical setting. This is convenient for the proof, as it is possible to explicitly construct a basis $\{f_{h,k,l}\}_{h,k\in\N, 0\leq l \leq l_h}$ of $L^2(B_3\setminus\overline B_2)$ with the special properties given in Lemma 2.1 of \cite{RS-Instability}, where $l_h$ is the number of spherical harmonics of order $h$ on $\partial B_1$. However, it would suffice to consider the exterior DN maps in any annulus $B_R\setminus \overline B_r$ with $1<r<R$, as the rest of the construction can be easily adapted to this case. Whether instability holds in the case of full data remains to be proved.
\end{remark}

\subsection{Organization of the article} The article is organized as follows. We begin Section~\ref{sec: preliminaries} by defining the many needed function spaces and recalling the notation for the fractional conductivity equation. Section~\ref{sec: sobolev lemmas} concerns extension and multiplication lemmas for Sobolev functions. In Section~\ref{sec: stability estimates}, we prove our main stability estimate, Theorem~\ref{thm: stability estimate}. In Section~\ref{sec: partial data reduction}, we discuss quantitative reduction to the Schrödinger problem with partial data. Finally, to completement the stability theorem, in Section~\ref{sec: instability} we prove the exponential instability result of Theorem \ref{thm: instability}. For clarity, some proofs of auxiliary results are postponed to Appendix~\ref{sec: Proofs of auxiliary lemmas}.

\subsection*{Acknowledgements} G.C. was supported by an Alexander-von-Humboldt postdoctoral fellowship. J.R. was supported by the Vilho, Yrjö and Kalle Väisälä Foundation of the Finnish Academy of Science and Letters.

\section{Preliminaries}\label{sec: preliminaries}

\subsection{Function spaces}\label{sec: Function spaces}

Throughout this article $\Omega\subset \R^n$ is always an open set. The classical Sobolev spaces of order $k\in\N$ and integrability exponent $p\in [1,\infty]$ are denoted by $W^{k,p}(\Omega)$ and for $k=0$ we use the convention $W^{0,p}(\Omega)=L^p(\Omega)$. Moreover, we let $W^{s,p}(\Omega)$ stand for the fractional Sobolev spaces, when $s\in \R_+\setminus\N$ and $1\leq p < \infty$. These spaces are also called Slobodeckij spaces or Gagliardo spaces. If $1\leq p<\infty$ and $s=k+\sigma$ with $k\in \N_0$, $0<\sigma<1$, then they are defined by
\[
    W^{s,p}(\Omega)\vcentcolon =\{\,u\in W^{k,p}(\Omega)\,;\, [\partial^{\alpha} u]_{W^{\sigma,p}(\Omega)}<\infty\quad \forall |\alpha|=k\, \},
\]
where 
\[
    [u]_{W^{\sigma,p}(\Omega)}\vcentcolon =\left(\int_{\Omega}\int_{\Omega}\frac{|u(x)-u(y)|^p}{|x-y|^{n+\sigma p}}\,dxdy\right)^{1/p}
\]
is the so-called Gagliardo seminorm. The Slobodeckij spaces are naturally endowed with the norm
\[
\|u\|_{W^{s,p}(\Omega)}\vcentcolon =\left(\|u\|_{W^{k,p}(\Omega)}^p+\sum_{|\alpha|=k}[\partial^{\alpha}u]_{W^{\sigma,p}(\Omega)}^p\right)^{1/p}.
\]

Next we recall the definition of the Bessel potential spaces $H^{s,p}(\R^n)$ and introduce several local variants of them. For the Fourier transform, we use the following convention
\[
    \fourier u(\xi)\vcentcolon = \hat u(\xi) \vcentcolon = \int_{\R^n} u(x)e^{-ix \cdot \xi} \,dx,
\]
whenever it is defined. Moreover, the Fourier transform acts as an isomorphism on the space of Schwartz functions $\schwartz(\R^n)$ and by duality on the space of tempered distributions $\tempered(\R^n)$. The inverse of the Fourier transform is denoted in each case by $\ifourier$. The Bessel potential of order $s \in \R$ is the Fourier multiplier $\vev{D}^s\colon \tempered(\R^n) \to \tempered(\R^n)$, that is
\begin{equation}\label{eq: Bessel pot}
    \vev{D}^s u \vcentcolon = \ifourier(\vev{\xi}^s\widehat{u}),
\end{equation} 
where $\vev{\xi}\vcentcolon = (1+|\xi|^2)^{1/2}$ is the Japanese-bracket. For any $s \in \R$ and $1 \leq p < \infty$, the Bessel potential space $H^{s,p}(\R^n)$ is defined by
\begin{equation}
\label{eq: Bessel pot spaces}
    H^{s,p}(\R^n) \vcentcolon = \{ u \in \tempered(\R^n)\,;\, \vev{D}^su \in L^p(\R^n)\},
\end{equation}
 which we endow with the norm $\norm{u}_{H^{s,p}(\R^n)} \vcentcolon = \norm{\vev{D}^su}_{L^p(\R^n)}$. If $\Omega\subset \R^n$, $F\subset\R^n$ are given open and closed sets, then we define the following local Bessel potential spaces:
\begin{equation}\label{eq: local bessel pot spaces}
\begin{split}
    \widetilde{H}^{s,p}(\Omega) &\vcentcolon = \mbox{closure of } C_c^\infty(\Omega) \mbox{ in } H^{s,p}(\R^n),\\
    H_F^{s,p}(\R^n) &\vcentcolon =\{\,u \in H^{s,p}(\R^n)\,;\, \supp(u) \subset F\,\},\\
    H^{s,p}(\Omega) &\vcentcolon = \{\,u|_{\Omega} \,;\, u \in H^{s,p}(\R^n)\,\}.
\end{split}
\end{equation}
The space $H^{s,p}(\Omega)$ is equipped with the quotient norm
\begin{equation}
\label{eq: norm local space}
    \norm{u}_{H^{s,p}(\Omega)} \vcentcolon = \inf\{\,\norm{w}_{H^{s,p}(\R^n)}\,;\, w \in H^{s,p}(\R^n), w|_\Omega = v\,\}.
\end{equation} 
We see that $\widetilde{H}^{s,p}(\Omega),H^{s,p}_F(\R^n)$ are closed subspaces of $H^{s,p}(\R^n)$. As customary, we set $H^s(\Omega)\vcentcolon =H^{s,2}(\Omega)$ for any open set $\Omega\subset \R^n$. If the boundary of the domain $\Omega\subset\R^n$ is regular enough then there is a close relation between the fractional Sobolev and Bessel potential spaces but also between two of the above introduced local Bessel potential spaces. For this purpose we next introduce the H\"older spaces and the notion of domains of class $C^{k,\alpha}$.

For all $k\in \N_0$ and $0<\alpha\leq 1$, the space $C^{k,\alpha}(\Omega)$ consists of all functions $u\in C^k(\Omega)$ such that the norm
\[
    \|u\|_{C^{k,\alpha}(\Omega)}\vcentcolon = \|u\|_{C^k(\Omega)}+\sum_{|\beta|=k}[\partial^{\beta}u]_{C^{0,\alpha}(\Omega)}
\]
is finite, where
\[
    \|u\|_{C^k(\Omega)}\vcentcolon =\sum_{|\beta|\leq k}\|\partial^{\beta}u\|_{L^{\infty}(\Omega)}\quad\text{and}\quad [u]_{C^{0,\alpha}(\Omega)}\vcentcolon =\sup_{x\neq y\in\Omega}\frac{|u(x)-u(y)|}{|x-y|^{\alpha}}.
\] 
We remark that the same notation will be used for $\R^m$-valued functions. We say that an open subset $\Omega\subset\R^n$ is of class $C^{k,\alpha}$ for $k\in\N_0$, $0<\alpha\leq 1$ if there exists $C>0$ such that for any $x\in\partial\Omega$ there exists a ball $B=B_r(x)$, $r>0$, and a map $T\colon Q\to B$ satisfying
\begin{enumerate}[(i)]
    \item $T\in C^{k,\alpha}(Q)$, $T^{-1}\in C^{k,\alpha}(B)$,
    \item $\|T\|_{C^{k,\alpha}(Q)}, \|T^{-1}\|_{C^{k,\alpha}(B)}\leq C$,
    \item $T(Q_+)=\Omega\cap B$, $T(Q_0)=\partial\Omega\cap B$.
\end{enumerate}
In the special case $k=0,\alpha=1$, we say that $\Omega$ is a Lipschitz domain. Moreover, we say that a domain is of class $C^k$ if the above conditions hold for $\alpha=0$ and it is smooth if it is of class $C^k$ for any $k\in\N$. Above we used the following notation:
\begin{align*}
    Q&\vcentcolon = \{x=(x',x_n)\in \R^{n-1}\times \R\,;\,|x'|<1,\,|x_n|<1\,\}\\
    Q_+&\vcentcolon = \{x=(x',x_n)\in Q\,;\,x_n>0\,\}\\
    Q_0&\vcentcolon = \{x=(x',x_n)\in Q\,;\,x_n=0\,\}
\end{align*}
One can prove the following equivalence of local Bessel potential spaces:

\begin{lemma}[{\cite[Theorem~3.29]{ML-strongly-elliptic-systems}}]
\label{lem: equivalence of local spaces}
    Let $\Omega\subset\R^n$ be a Lipschitz domain with bounded boundary and $s\in\R$ then $\widetilde{H}^s(\Omega)=H^s_{\overline{\Omega}}(\R^n)$.
\end{lemma}

Next we note that the following embeddings hold between Bessel potential spaces $H^{s,p}$ and the fractional Sobolev spaces $W^{s,p}$:

\begin{theorem}
\label{thm: embeddings between Bessel and Gagliardo}
    Let $s\in \R_+\setminus\N$, $1<p<\infty$ and assume $\Omega\subset\R^n$ is an open set.
    \begin{enumerate}[(i)]
        \item\label{item: embedding p < 2} If $1<p\leq 2$, $s=k+\sigma$ with $k\in\N_0$, $0<\sigma<1$ and $\Omega=\R^n$ or $\Omega$ is of class $C^{k,1}$ with bounded boundary, then $W^{s,p}(\Omega)\hookrightarrow H^{s,p}(\Omega)$.
        \item\label{item: embedding p > 2} If $2\leq p< \infty$, then $H^{s,p}(\Omega)\hookrightarrow W^{s,p}(\Omega)$.
    \end{enumerate}
\end{theorem}
\begin{remark}
    In the range $0<s<1$, this theorem is a standard result. For $\Omega=\R^n$ a proof can be found in \cite[Chapter~V, Theorem~5]{Singular-Integrals-Stein} and by using the extension theorem in Slobodeckij spaces (cf.~\cite[Theorem~5.4]{DINEPV-hitchhiker-sobolev}) it follows for Lipschitz domains with bounded boundary. In the higher order case $s>1$ it seems to be less well-known and therefore we provide a proof in the Appendix~\ref{sec: Proofs of auxiliary lemmas}.
\end{remark}
\begin{remark}
    In particular, the above theorem asserts that for all $s=k+\sigma$ with $k\in\N_0, 0<\sigma<1$ there holds $H^s(\Omega)=W^{s,2}(\Omega)$, when $\Omega\subset\R^n$ is an open set of class $C^{k,1}$ with bounded boundary.
\end{remark}

\subsection{Fractional Laplacians, fractional gradient and fractional divergence}

For all $s\geq 0$ and $u\in\tempered(\R^n)$, we define the fractional Laplacian of order $s$ by
\[
    (-\Delta)^su\vcentcolon = \ifourier(|\xi|^{2s}\widehat{u}),
\]
whenever the right hand side is well-defined. One can easily show by using the Mikhlin multiplier theorem that the fractional Laplacian is a bounded linear operator $(-\Delta)^{s}\colon H^{t,p}(\R^n) \to H^{t-2s,p}(\R^n)$ for all $t\in\R$ and $1< p<\infty$. In the special case $u\in \schwartz(\R^n)$ and $s\in(0,1)$, the fractional Laplacian can be calculated as the following singular integrals (see e.g.~\cite[Section 3]{DINEPV-hitchhiker-sobolev})
\begin{equation}
\begin{split}
    (-\Delta)^su(x)&=C_{n,s}\,\text{p.v.}\int_{\R^n}\frac{u(x)-u(y)}{|x-y|^{n+2s}}\,dy\\
    &= -\frac{C_{n,s}}{2}\int_{\R^n}\frac{u(x+y)+u(x-y)-2u(x)}{|y|^{n+2s}}\,dy,
\end{split}
\end{equation}
where $C_{n,s}>0$ is a normalization constant. One immediately sees that the above integral is in the range $s\in (0,1/2)$ for (local) Lipschitz functions not really singular. The fractional Laplacian has a distinguished property which simplifies the analysis of the inverse fractional conductivity problem compared to the classical Calder\'on problem, namely the unique continuation property (UCP), which asserts that if $r\in\R$, $1\leq p<\infty$, $s\in\R_+\setminus\N$ and $u\in H^{r,p}(\R^n)$ satisfies $u|_{V}=(-\Delta)^su|_V=0$ in some nonempty open set $V\subset \R^n$, then there holds $u\equiv 0$ in $\R^n$ (cf.~\cite[Theorem~2.2]{KRZ2022Biharm}).

Moreover, let us point out that a large part of the theory of the inverse fractional conductivity problem can be extended to a certain class of unbounded domains, which are called domains bounded in one direction (cf.~\cite[Definition 2.1]{RZ2022unboundedFracCald}), since the fractional Laplacian satisfies on these domains a Poincar\'e inequality. But in this work, we restrict our attention to bounded domains and therefore the stablity of the inverse fractional conductivity problem on these domains is still open.
    
For the rest of this section, we fix $s\in(0,1)$. The fractional gradient of order $s$ is the bounded linear operator $\nabla^s\colon H^s(\R^n)\to L^2(\R^{2n};\R^n)$ given by (see ~\cite{covi2019inverse-frac-cond,DGLZ12,RZ2022unboundedFracCald})
    \[
        \nabla^su(x,y)\vcentcolon =\sqrt{\frac{C_{n,s}}{2}}\frac{u(x)-u(y)}{|x-y|^{n/2+s+1}}(x-y),
    \]
    and satisfies
    \begin{equation}
    \label{eq: bound on fractional gradient}
        \|\nabla^su\|_{L^2(\R^{2n})}=\|(-\Delta)^{s/2}u\|_{L^2(\R^n)}\leq \|u\|_{H^s(\R^n)}
    \end{equation}
    for all $u\in H^s(\R^n)$. The adjoint of $\nabla^s$ is called fractional divergence of order $s$ and denoted by $\Div_s$. More concretely, the fractional divergence of order $s$ is the bounded linear operator 
    \[
        \Div_s\colon L^2(\R^{2n};\R^n)\to H^{-s}(\R^n)
    \] 
    satisfying the identity
    \[
        \langle \Div_su,v\rangle_{H^{-s}(\R^n)\times H^s(\R^n)}=\langle u,\nabla^sv\rangle_{L^2(\R^{2n})}
    \]
    for all $u\in L^2(\R^{2n};\R^n),v\in H^s(\R^n)$. A simple estimate shows that there holds (see ~\cite[Section 8]{RZ2022unboundedFracCald})
    \[
        \|\Div_s(u)\|_{H^{-s}(\R^n)}\leq \|u\|_{L^2(\R^{2n})}
    \]
    for all $u\in L^2(\R^{2n};\R^n)$, and also a comparison with the quadratic form definition for the fractional Laplacian implies $(-\Delta)^su=\Div_s(\nabla^su)$ for all $u\in H^s(\R^n)$ (see~e.g.~\cite[Theorem~1.1]{KWA-ten-definitions-fractional-laplacian} and \cite[Lemma 2.1]{covi2019inverse-frac-cond}). 

\section{Extension and multiplication lemmas for fractional Sobolev spaces}\label{sec: sobolev lemmas}
\label{sec: Higher order extension, multiplication theorems and Gagliardo--Nirenberg inequalities in Slobodeckij spaces on domains}

In this section, we establish a higher extension theorem for the spaces $W^{s,p}(\Omega)$, where $\Omega\subset\R^n$ is a sufficiently regular domain with bounded boundary. This is then used to extend a Gagliardo--Nirenberg inequality to these domains (cf.~\cite[Corollary~2, (iii)]{BrezisComposition}). These results are needed to have access to suitable Hölder embeddings for the conductivities and have access to $L^\infty$ estimates in $\Omega_e$ for the conductivities via concrete extension operators and the Gagliardo--Nirenberg inequality. This need in turn is related to having only $L^\infty$ control of conductivities in the exterior via Theorem \ref{thm: exterior stability}. The proofs of Lemmas \ref{lem: multiplication Hoelder} and \ref{lem: zero extension} are found in Appendix \ref{sec: Proofs of auxiliary lemmas}.

\begin{lemma}[Multiplication by H\"older functions]
\label{lem: multiplication Hoelder}
    Let $\Omega\subset \R^n$ be a Lip-schitz domain with bounded boundary, $1\leq p<\infty$ and $s\in \R_+\setminus \N$.
    \begin{enumerate}[(i)]
        \item\label{item: basic case} Let $0<s<1$ and $s<\mu\leq 1$. If $u\in W^{s,p}(\Omega)$ and $\phi\in C^{0,\mu}(\Omega)$, then $\phi u\in W^{s,p}(\Omega)$ satisfies
        \begin{equation}
        \label{eq: estimate basic case}
            \|\phi u\|_{W^{s,p}(\Omega)}\leq C_1\left(1+\left(\frac{\mu}{s(\mu-s)}\right)^{1/p}\right) \|\phi\|_{C^{0,\mu}(\Omega)}\|u\|_{W^{s,p}(\Omega)}
        \end{equation}
        for some $C_1>0$ only depending on $n$ and $p$.
        \item\label{item: higher order case} Let $s=k+\sigma$ with $k\in\N$, $0<\sigma<1$ and $\sigma<\mu\leq 1$. If $u\in W^{s,p}(\Omega)$ and $\phi \in C^k(\Omega)$ with $\partial^{\alpha}\phi \in C^{0,\mu}(\Omega)$ for all $|\alpha|\leq k$, then $\phi u\in W^{s,p}(\Omega)$ satisfying
        \begin{equation}
        \label{eq: estimate higher order case}
            \begin{split}
                &\|\phi u\|_{W^{s,p}(\Omega)}\leq C_1C_2\left(\sum_{\ell=0}^k\|\nabla^{\ell}\phi\|_{C^{0,\mu}(\Omega)}\right)\|u\|_{W^{s,p}(\Omega)}
            \end{split}
        \end{equation}
        for some $C_2>0$ only depending on $n,k,p,\Omega$.
    \end{enumerate}
\end{lemma}

The following statement is a generalization of \cite[Lemma~5.1]{DINEPV-hitchhiker-sobolev}, where the support of $u$ is not necessarily compact and $s$ is allowed to be larger than one:

\begin{lemma}[Zero extension]
\label{lem: zero extension}
    Let $\Omega\subset \R^n$ be an open set, $s>0$ and $1\leq p<\infty$. Assume that $u\in W^{s,p}(\Omega)$ satisfies $d\vcentcolon =\dist(\supp(u),\partial \Omega)>0$ and let $\Bar{u}\colon \R^n\to \R$ be its zero extension, that is
    \begin{equation}
       \Bar{u}(x)\vcentcolon = \begin{cases}
            u(x),&\quad x\in\Omega\\
            0,&\quad\text{otherwise}.
        \end{cases}
    \end{equation}
    Then $\Bar{u}\in W^{s,p}(\R^n)$ and there holds
    \begin{equation}
        \|\Bar{u}\|_{W^{s,p}(\R^n)}\leq C\|u\|_{W^{s,p}(\Omega)}.
    \end{equation}
\end{lemma}

\begin{lemma}[Higher order extension theorem]
\label{lem: extension}
    Let $1\leq p<\infty$, $s=k+\sigma$ with $k\in\N_0$, $0<\sigma<1$ and assume that $\Omega\subset\R^n$ is a domain of class $C^{k,1}$ with bounded boundary. Then there exists an extension operator $E\colon W^{s,p}(\Omega)\to W^{s,p}(\R^n)$ such that $Eu|_{\Omega}=u$ and $\|Eu\|_{W^{s,p}(\R^n)}\leq C\|u\|_{W^{s,p}(\Omega)}$.
\end{lemma}
\begin{proof}[Proof of Lemma \ref{lem: extension}]
    By assumption there is a finite collection of balls $B_j$, $j=1,\ldots,m$, and maps $T_j\colon Q\to B_j$ such that
    \begin{enumerate}[(i)]
        \item $T_j\in C^{k,1}(Q)$, $T_j^{-1}\in C^{k,1}(B_j)$,
        \item $\|T_j\|_{C^{k,1}(Q)},\|T_j^{-1}\|_{C^{k,1}(B_j)}\leq C$ for some $C>0$,
        \item $T_j(Q_+)=\Omega\cap B_j$, $T_j(Q_0)=\partial\Omega\cap B_j$
    \end{enumerate}
    for all $j=1,\ldots,m$. By \cite[Lemma~9.3]{FA-Brezis} there exist $(\phi_j)_{j=0,\ldots,m}\subset C^{\infty}(\R^n)$ such that
    \begin{enumerate}[(I)]
        \item\label{item: part of unity 1} $0\leq \phi_j\leq 1$ for all $j=0,\ldots,m,$
        \item\label{item: part of unity 2} $\supp(\phi_0)\subset\R^n\setminus\partial\Omega$,
        \item\label{item: part of unity 3} $\phi_j\in C_c^{\infty}(B_j)$ for all $j=1,\ldots,m$
        \item\label{item: part of unity 4} and $\sum_{j=0}^m\phi_j=1$ on $\R^n$.
    \end{enumerate}
    Using the compactness of $\partial\Omega$ and the assertion \ref{item: part of unity 2} we see that 
    \[
        d\vcentcolon = \dist(\supp(\phi_0|_{\Omega}),\partial\Omega)>0.
    \] 
    On the other hand the properties \ref{item: part of unity 3}, \ref{item: part of unity 4} imply $\partial^{\alpha}\phi_0\in C^{0,1}(\Omega)$ for all $\alpha\in \N^n_0$. Hence, by Lemma~\ref{lem: multiplication Hoelder} we know $\phi_0u\in W^{s,p}(\Omega)$ and therefore we deduce from Lemma~\ref{lem: zero extension} that $u_0\vcentcolon = \overline{\phi_0u}\in W^{s,p}(\R^n)$. Next we want to extend the functions $\phi_j u$ to elements of $W^{s,p}(\R^n)$. In the proof of \cite[Satz~6.10, Satz~6.38]{dobrowolski2010angewandte}, which establishes the result for bounded domains, it has been shown that there exists $u_j\in W^{s,p}(\R^n)$ such that $u_j|_{\Omega}=\phi_j u$ for all $j=1,\ldots,m$ and $\|u_j\|_{W^{s,p}(\R^n)}\leq C\|u\|_{W^{s,p}(\Omega)}$. Therefore, the operator $E\colon W^{s,p}(\Omega)\to W^{s,p}(\R^n)$ given by $Eu\vcentcolon = \sum_{j=0}^{m}u_j$
    satisfies the asserted properties and we can conclude the proof.
\end{proof}

\begin{lemma}[Gagliardo--Nirenberg inequality]
\label{lem: Gagliardo Nirenberg inequality}
    Let $1< p<\infty$, $s=k+\sigma$ with $k\in\N_0$, $0<\sigma<1$ and assume that $\Omega=\R^n$ or $\Omega\subset\R^n$ is a domain of class $C^{k,1}$ with bounded boundary. Then for any $0< \theta < 1$ there holds
    \begin{equation}
    \label{eq: Gagliardo Nirenberg on domains}
        \|u\|_{W^{\theta s, p/\theta}(\Omega)}\leq C\|u\|_{W^{s,p}(\Omega)}^{\theta}\|u\|_{L^{\infty}(\Omega)}^{1-\theta}
    \end{equation}
    for all $u\in W^{s,p}(\Omega)\cap L^{\infty}(\Omega)$.
\end{lemma}

\begin{proof}
    In the case $\Omega=\R^n$ the result holds by \cite[Corollary~2.c)]{BrezisComposition}. If $\Omega\subset\R^n$ is a domain of class $C^{k,1}$ with bounded boundary then by Lemma~\ref{lem: extension} for all $u\in W^{s,p}(\Omega)\cap L^{\infty}(\Omega)$ there is an extension $Eu\in W^{s,p}(\R^n)$. Moreover, the proof in \cite[Satz~6.10, Satz~6.38]{dobrowolski2010angewandte} shows that one has $\|Eu\|_{L^{\infty}(\R^n)}\leq C\|u\|_{L^{\infty}(\Omega)}$ as the extensions $u_j$ are obtained by a higher order reflection technique. Therefore, we deduce
    \[
    \begin{split}
        \|u\|_{W^{\theta s, p/\theta}(\Omega)}&\leq \|Eu\|_{W^{\theta s, p/\theta}(\R^n)}\leq \|Eu\|_{W^{s,p}(\R^n)}^{\theta}\|Eu\|_{L^{\infty}(\R^n)}^{1-\theta}\\
        &\leq C\|u\|_{W^{s,p}(\Omega)}^{\theta}\|u\|_{L^{\infty}(\Omega)}^{1-\theta}.
    \end{split}
    \]
    Hence, we can conclude the proof.
\end{proof}

\section{Stability estimates}\label{sec: stability estimates}

We prove Theorem~\ref{thm: stability estimate} in this section. In the proof, we make use of the exterior determination result stated in Theorem \ref{thm: exterior stability}. Then we establish H\"older estimates for the function $\gamma_1^{-1/2}-\gamma_2^{-1/2}$ in terms of $\gamma_1^{1/2}-\gamma_2^{1/2}$ and a quantitative version of \cite[Corollary~3.6]{RZ2022LowReg}. Afterwards, we prove a reduction theorem, which demonstrates that the difference of the DN maps corresponding to two potentials $q_{_1},q_{2}$ can essentially be controlled by powers of the difference of the DN maps related to the conductivities $\gamma_1, \gamma_2$. Finally, using the stability result stated in Theorem \ref{thm: ruland salo proposition} for the fractional Schr\"odinger opeorators, we can prove Theorem~\ref{thm: stability estimate}.

\subsection{Reduction Lemma}
\label{subsec: reduction lemma}

\begin{lemma}
\label{lem: Holder regularity}
    Let $\Omega\subset\R^n$ be an open set and $0<\alpha\leq 1$. For all $\gamma_1,\gamma_2\in L^{\infty}(\Omega)$ satisfying $\gamma_1(x),\gamma_2(x)\geq \gamma_0>0$, we have
    \begin{equation}
    \label{eq: L infty bound}
    \|\gamma_1^{-1/2}-\gamma_2^{-1/2}\|_{L^{\infty}(\Omega)}\leq C\|\gamma_1^{1/2}-\gamma_2^{1/2}\|_{L^{\infty}(\Omega)}\leq C\|\gamma_1-\gamma_2\|_{L^{\infty}(\Omega)}^{1/2}.
    \end{equation}
    Moreover, under the additional assumption $\gamma_1^{1/2},\gamma_2^{1/2}\in C^{0,\alpha}(\Omega)$, there holds $\gamma_i^{-1/2}\in C^{0,\alpha}(\Omega)$ with
    \begin{equation}
    \label{eq: single Hölder bound}
        \|\gamma_i^{-1/2}\|_{L^{\infty}(\Omega)}\leq 1/\gamma_0^{1/2},\quad  [\gamma_i^{-1/2}]_{C^{0,\alpha}(\Omega)}\leq \frac{[\gamma^{1/2}_i]_{C^{0,\alpha}(\Omega)}}{\gamma_0}
    \end{equation}
    for $i=1,2$ and $\gamma_1^{-1/2}-\gamma_2^{-1/2}\in C^{0,\alpha}(\Omega)$ satisfying
    \begin{equation}
    \label{eq: Holder bound}
        \begin{split}
            &[\gamma_1^{-1/2}-\gamma_2^{-1/2}]_{C^{0,\alpha}(\Omega)}\leq  \frac{[\gamma_1^{1/2}-\gamma_2^{1/2}]_{C^{0,\alpha}(\Omega)}}{\gamma_0}\\
        &+\frac{\|\gamma_1^{1/2}-\gamma_2^{1/2}\|_{L^{\infty}(\Omega)}}{\gamma_0^{3/2}}([\gamma_2^{1/2}]_{C^{0,\alpha}(\Omega)}+[\gamma_1^{1/2}]_{C^{0,\alpha}(\Omega)}).
        \end{split}
    \end{equation}
\end{lemma}

\begin{proof}
    We have
    \[
        \begin{split}
            |\gamma_1^{-1/2}(x)-\gamma_2^{-1/2}(x)|&=\left|\frac{\gamma_2^{1/2}(x)-\gamma_1^{1/2}(x)}{\gamma_1^{1/2}(x)\gamma_2^{1/2}(x)}\right|\leq \gamma_0^{-1}\|\gamma^{1/2}_1-\gamma^{1/2}_2\|_{L^{\infty}(\Omega)}
        \end{split}
    \]
    for all $x\in \Omega$ and therefore the first estimate in \eqref{eq: L infty bound} follows. The second part in \eqref{eq: L infty bound} follows from the estimate $|a^{1/2}-b^{1/2}|\leq |a-b|^{1/2}$
    for all $a,b\in \R_+$.
    
   From now on assume that the functions $\gamma_1,\gamma_2$ satisfy additionally $\gamma_1^{1/2},\gamma_2^{1/2}\in C^{0,\alpha}(\Omega)$. Using the uniform ellipticity of $\gamma_1,\gamma_2$, we have $\|\gamma_i^{-1/2}\|_{L^{\infty}(\Omega)}\leq\gamma_0^{-1/2} $ and
    \[
        |\gamma_i^{-1/2}(x)-\gamma_i^{-1/2}(y)|=\frac{|\gamma_i^{1/2}(y)-\gamma^{1/2}(x)|}{|\gamma_i^{1/2}(x)\gamma_i^{1/2}(y)|}\leq \frac{[\gamma^{1/2}_i]_{C^{0,\alpha}(\Omega)}}{\gamma_0}|x-y|^{\alpha}.
    \]
    This establishes the estimate \eqref{eq: single Hölder bound} and hence $\gamma_i\in C^{0,\alpha}(\Omega)$ for $i=1,2$. Next we prove the bound \eqref{eq: Holder bound}. We have
    \[
        [\gamma_1^{-1/2}-\gamma_2^{-1/2}]_{C^{0,\alpha}(\Omega)}=\left[\frac{\gamma_2^{1/2}-\gamma_1^{1/2}}{\gamma_1^{1/2}\gamma_2^{1/2}}\right]_{C^{0,\alpha}(\Omega)}=\left[\frac{\gamma_1^{1/2}-\gamma_2^{1/2}}{\gamma_1^{1/2}\gamma_2^{1/2}}\right]_{C^{0,\alpha}(\Omega)}.
    \]
    We have
    \[
    \begin{split}
        &\left|\frac{(\gamma_1^{1/2}-\gamma_2^{1/2})(x)}{\gamma_1^{1/2}(x)\gamma_2^{1/2}(x)}-\frac{(\gamma_1^{1/2}-\gamma_2^{1/2})(y)}{\gamma_1^{1/2}(y)\gamma_2^{1/2}(y)}\right|\\
        &=\frac{|(\gamma_1^{1/2}-\gamma_2^{1/2})(x)\gamma_1^{1/2}(y)\gamma_2^{1/2}(y)-(\gamma_1^{1/2}-\gamma_2^{1/2})(y)\gamma_1^{1/2}(x)\gamma_2^{1/2}(x)|}{\gamma_1^{1/2}(x)\gamma_2^{1/2}(x)\gamma_1^{1/2}(y)\gamma_2^{1/2}(y)}\\
        &=\left|\frac{((\gamma_1^{1/2}-\gamma_2^{1/2})(x)-(\gamma_1^{1/2}-\gamma_2^{1/2})(y))\gamma_1^{1/2}(y)\gamma_2^{1/2}(y)|}{\gamma_1^{1/2}(x)\gamma_2^{1/2}(x)\gamma_1^{1/2}(y)\gamma_2^{1/2}(y)}\right.\\
        &\quad\,+\left.\frac{(\gamma_1^{1/2}-\gamma_2^{1/2})(y)(\gamma_1^{1/2}(y)\gamma_2^{1/2}(y)-\gamma_1^{1/2}(x)\gamma_2^{1/2}(x))}{\gamma_1^{1/2}(x)\gamma_2^{1/2}(x)\gamma_1^{1/2}(y)\gamma_2^{1/2}(y)}\right|
    \end{split}
    \]
    for all $x,y\in\Omega$. Next observe that there holds
    \[
    \begin{split}
        &\gamma_1^{1/2}(y)\gamma_2^{1/2}(y)-\gamma_1^{1/2}(x)\gamma_2^{1/2}(x)\\
       &=-(\gamma_1^{1/2}(y)(\gamma_2^{1/2}(x)-\gamma_2^{1/2}(y))+\gamma_2^{1/2}(x)(\gamma_1^{1/2}(x)-\gamma_1^{1/2}(y))).
    \end{split}
    \]
    By assumption we get
    \[
    \begin{split}
        &\left|\frac{(\gamma_1^{1/2}-\gamma_2^{1/2})(x)}{\gamma_1^{1/2}(x)\gamma_2^{1/2}(x)}-\frac{(\gamma_1^{1/2}-\gamma_2^{1/2})(y)}{\gamma_1^{1/2}(y)\gamma_2^{1/2}(y)}\right|\\
        &\leq\frac{|(\gamma_1^{1/2}-\gamma_2^{1/2})(x)-(\gamma_1^{1/2}-\gamma_2^{1/2})(y)|}{\gamma_1^{1/2}(x)\gamma_2^{1/2}(x)}\\
        &\quad +|(\gamma_1^{1/2}-\gamma_2^{1/2})(y)|\left(\frac{|\gamma_2^{1/2}(x)-\gamma_2^{1/2}(y)|}{\gamma_1^{1/2}(x)\gamma_2^{1/2}(x)\gamma_2^{1/2}(y)}+\frac{|\gamma_1^{1/2}(x)-\gamma_1^{1/2}(y)|}{\gamma_1^{1/2}(x)\gamma_1^{1/2}(y)\gamma_2^{1/2}(y)}\right)\\
        &\leq \left(\frac{[\gamma_1^{1/2}-\gamma_2^{1/2}]_{C^{0,\alpha}(\Omega)}}{\gamma_0}+\frac{\|\gamma_1^{1/2}-\gamma_2^{1/2}\|_{L^{\infty}(\Omega)}}{\gamma_0^{3/2}}([\gamma_2^{1/2}]_{C^{0,\alpha}(\Omega)}+[\gamma_1^{1/2}]_{C^{0,\alpha}(\Omega)})\right)\\
        &\quad\cdot|x-y|^{\alpha}
    \end{split}
    \]
    for all $x,y\in\Omega$ and hence there holds
    \[
    \begin{split}
        &[\gamma_1^{-1/2}-\gamma_2^{-1/2}]_{C^{0,\alpha}(\Omega)}\\
        &\leq  \frac{[\gamma_1^{1/2}-\gamma_2^{1/2}]_{C^{0,\alpha}(\Omega)}}{\gamma_0}+\frac{\|\gamma_1^{1/2}-\gamma_2^{1/2}\|_{L^{\infty}(\Omega)}}{\gamma_0^{3/2}}([\gamma_2^{1/2}]_{C^{0,\alpha}(\Omega)}+[\gamma_1^{1/2}]_{C^{0,\alpha}(\Omega)}).
    \end{split}
    \]
\end{proof}

\begin{lemma}[Multiplication by Sobolev functions]
\label{lem: multiplication sobolev}
    Let $\Omega\subset \R^n$ be an open set and $0<s<1$. If $u\in H^s(\Omega)$ and $\gamma\in L^{\infty}(\R^n)$ with background deviation $m\in H^{s,n/s}(\R^n)$ satisfies $\gamma(x)\geq \gamma_0>0$ then there holds
    \begin{equation}
    \label{eq: gamma 1/2 estimate}
        \|\gamma^{1/2}u\|_{H^s(\Omega)}\leq C(1+\|m\|_{L^{\infty}(\R^n)}+\|m\|_{H^{s,n/s}(\R^n)})\|u\|_{H^s(\Omega)}
    \end{equation}
    and
    \begin{equation}
    \label{eq: gamma -1/2 estimate}
         \|\gamma^{-1/2}u\|_{H^s(\Omega)}\leq C(1+\|m\|_{L^{\infty}(\R^n)}+\|m\|_{H^{s,n/s}(\R^n)})\|u\|_{H^s(\Omega)}.
    \end{equation}
\end{lemma}

\begin{proof} 
Let $Eu \in H^s(\R^n)$ be an extension of $u$ such that $\norm{Eu}_{H^s(\R^n)} \leq 2\norm{u}_{H^s(\Omega)}$. This extension exists by the quotient space definition of $H^s(\Omega)$.
Thus, applying \cite[Lemma~3.4]{RZ2022LowReg} to $Eu\in H^s(\R^n)$, we deduce
    \[
    \begin{split}
        \|\gamma^{1/2}u\|_{H^s(\Omega)}&\leq \|\gamma^{1/2}Eu\|_{H^s(\R^n)}\leq \|mEu\|_{H^s(\R^n)}+\|Eu\|_{H^s(\R^n)}\\
        &\leq C(1+\|m\|_{L^{\infty}(\R^n)}+\|m\|_{H^{s,n/s}(\R^n)})\|Eu\|_{H^s(\R^n)}\\
        &\leq C(1+\|m\|_{L^{\infty}(\R^n)}+\|m\|_{H^{s,n/s}(\R^n)})\|u\|_{H^s(\Omega)}.
    \end{split}
    \]
    This establishes \eqref{eq: gamma 1/2 estimate}. Arguing as in the proof of \cite[Lemma~3.7]{RZ2022LowReg} we can write $\gamma^{-1/2}=1-\frac{m}{m+1}$ with $\frac{m}{m+1}\in H^{s,n/s}(\R^n)$ and $\|\frac{m}{m+1}\|_{H^{s,n/s}(\R^n)}\leq \|m\|_{H^{s,n/s}(\R^n)}$. Thus, we can repeat the above estimates to obtain \eqref{eq: gamma -1/2 estimate}.
\end{proof}

\begin{theorem}
\label{thm: reduction}
    Let $0<s<\min(1,n/2)$, $\theta_0\in (s/n,1)$, $0<\epsilon\ll 1$ and $k\in\N_0$ satisfy
    \begin{equation}
    \label{eq: conditions exponents}
        k< \frac{2s+\epsilon}{\theta_0}<k+1\quad\text{and}\quad \ell s+\epsilon\notin\N\quad\forall \ell = 1,2.
    \end{equation} 
    Assume that $\Omega\subset \R^n$ is a domain of class $C^{k,1}$ with bounded boundary and the conductivities $\gamma_1,\gamma_2\in L^{\infty}(\R^n)$ with background deviations $m_1,m_2$ and potentials $q_1,q_2$ fulfill the following conditions:
    \begin{enumerate}[(i)]
        \item\label{item: assumption 1} $\gamma_0\leq \gamma_1(x),\gamma_2(x)\leq \gamma_0^{-1}$ for some $0<\gamma_0<1$,
        \item\label{item: assumption 2} $m_1,m_2\in  H^{s,n/s}(\R^n)\cap W^{2s+\epsilon,n/s}(\Omega_e)$ with $m_1-m_2\in W^{\frac{2s+\epsilon}{\theta_0},\theta_0 n/s}(\Omega_e)$,
        \item\label{item: assumption 3} there exist $C_1,C_2,C_{\theta_0}>0$ such that 
        \begin{equation}
        \label{eq: boundedness assumptions}
            \|m_i\|_{H^{s,n/s}(\R^n)}\leq C_1,\quad \|m_i\|_{W^{2s+\epsilon,n/s}(\Omega_e)} \leq C_2
        \end{equation}
        for $i=1,2$ and
        \begin{equation}
        \label{eq: difference bound}
            \|m_1-m_2\|^{\theta_0}_{W^{\frac{2s+\epsilon}{\theta_0},\theta_0 n/s}(\Omega_e)}\leq C_{\theta_0}.
        \end{equation}
    \end{enumerate} 
    Then there holds
    \begin{equation}
    \label{eq: reduction}
        \|\Lambda_{q_1}-\Lambda_{q_2}\|_*\leq CC_{\theta_0}(\|\Lambda_{\gamma_1}-\Lambda_{\gamma_2}\|_*+\|\Lambda_{\gamma_1}-\Lambda_{\gamma_2}\|_*^{\frac{1}{2}}+\|\Lambda_{\gamma_1}-\Lambda_{\gamma_2}\|_*^{\frac{1-\theta_0}{2}}).
    \end{equation}
\end{theorem}

\begin{proof}
    Let $f,g\in H^s(\Omega_e)$ and for $i=1,2$ denote by $v^{i}_f\in H^s(\R^n)$ the unique solution to the fractional Schr\"odinger equation $(-\Delta)^s+q_i$ (see \cite[Lemma 3.11]{RZ2022LowReg}). Using Lemma~\ref{lemma: Liouville reduction},  we deduce for $i=1,2$ and any extension $e_g\in H^s(\R^n)$ of $g\in H^s(\Omega_e)$ the identity
    \[
        \langle \Lambda_{q_i}f,g\rangle= B_{q_i}(v_f^{i},e_g)=B_{\gamma_i}(\gamma_i^{-1/2}v_f^{i}, \gamma_{i}^{-1/2}e_g)=\langle \Lambda_{\gamma_i}(\gamma_i^{-1/2}f),\gamma_i^{-1/2}g\rangle.
    \]
    Therefore, we obtain
    \[
    \begin{split}
        &\langle (\Lambda_{q_1}-\Lambda_{q_2})f,g\rangle = \langle \Lambda_{\gamma_1}(\gamma_1^{-1/2}f),\gamma_1^{-1/2}g\rangle -\langle \Lambda_{\gamma_2}(\gamma_2^{-1/2}f),\gamma_2^{-1/2}g\rangle \\
        &=\,\langle \Lambda_{\gamma_1}(\gamma_1^{-1/2}f),(\gamma_1^{-1/2}-\gamma_2^{-1/2})g\rangle +\langle \Lambda_{\gamma_1}(\gamma_1^{-1/2}f),\gamma_2^{-1/2}g\rangle \\
        &\quad \,\,-\langle \Lambda_{\gamma_2}(\gamma_2^{-1/2}-\gamma_1^{-1/2})f,\gamma_2^{-1/2}g\rangle -\langle \Lambda_{\gamma_2}(\gamma_1^{-1/2}f),\gamma_2^{-1/2}g\rangle\\
        &=\,\langle \Lambda_{\gamma_1}(\gamma_1^{-1/2}f),(\gamma_1^{-1/2}-\gamma_2^{-1/2})g\rangle+\langle (\Lambda_{\gamma_1}-\Lambda_{\gamma_2})(\gamma_1^{-1/2}f),\gamma_2^{-1/2}g\rangle\\
        &\quad \,\, +\langle \Lambda_{\gamma_2}(\gamma_1^{-1/2}-\gamma_2^{-1/2})f,\gamma_2^{-1/2}g\rangle\\
        &=\colon I_1+I_2+I_3
    \end{split}
    \]
    for all $f,g\in H^s(\Omega_e)$. Next note that the assumption \ref{item: assumption 1} and the fact that solutions to the homogeneous fractional conductivity equation depend continuously on the data imply
    \begin{equation}
    \label{eq: DN maps bounded}
        \|\Lambda_{\gamma_i}\|_*\leq C
    \end{equation}
    for $i=1,2$ and some $C>0$. On the other hand, using Lemma~\ref{lem: multiplication sobolev}, the uniform ellipticity \ref{item: assumption 1} and the uniform bound \eqref{eq: boundedness assumptions} , we deduce
    \begin{equation}
    \label{eq: boundedness sqrt gamma and inv sqrt gamma}
        \|\gamma^{1/2}f\|_{H^s(\Omega_e)}\leq C\|f\|_{H^s(\Omega_e)}\quad \text{and}\quad \|\gamma^{-1/2}f\|_{H^s(\Omega_e)}\leq C\|f\|_{H^s(\Omega_e)}
    \end{equation}
    for all $f\in H^s(\Omega_e)$ and some $C>0$. Using \eqref{eq: DN maps bounded}, \eqref{eq: boundedness sqrt gamma and inv sqrt gamma} and Lemma~\ref{lem: multiplication Hoelder}, we can estimate $I_1$ as follows:
    \begin{equation}
    \label{eq: estimate I1}
        \begin{split}
            |I_1|&\leq \|\Lambda_{\gamma_1}\|_*\|\gamma_1^{-1/2}f\|_{H^s(\Omega_e)}\|(\gamma_1^{-1/2}-\gamma_2^{-1/2})g\|_{H^s(\Omega_e)}\\
            &\leq \|\Lambda_{\gamma_1}\|_*\|\gamma_1^{-1/2}f\|_{H^s(\Omega_e)}\|\gamma_1^{-1/2}-\gamma_2^{-1/2}\|_{C^{0,s+\epsilon}(\Omega_e)}\|g\|_{H^s(\Omega_e)}\\
            &\leq C\|\gamma_1^{-1/2}-\gamma_2^{-1/2}\|_{C^{0,s+\epsilon}(\Omega_e)}\|f\|_{H^s(\Omega_e)}\|g\|_{H^s(\Omega_e)}.
        \end{split}
    \end{equation}
    By Lemma~\ref{lem: Holder regularity} we can upper bound the H\"older norm by
    \begin{equation}
    \label{eq: difference I1 }
    \begin{split}
        &\|\gamma_1^{-1/2}-\gamma_2^{-1/2}\|_{C^{0,s+\epsilon}(\Omega_e)}\leq C\|\gamma_1-\gamma_2\|_{L^{\infty}(\Omega_e)}^{1/2}\\
        &+\frac{[\gamma_1^{1/2}-\gamma_2^{1/2}]_{C^{0,s+\epsilon}(\Omega_e)}}{\gamma_0}+\frac{\|\gamma_1-\gamma_2\|_{L^{\infty}(\Omega_e)}^{1/2}}{\gamma_0^{3/2}}([\gamma_2^{1/2}]_{C^{0,s+\epsilon}(\Omega_e)}+[\gamma_1^{1/2}]_{C^{0,s+\epsilon}(\Omega_e)})\\
        &\leq C(1+[\gamma_2^{1/2}]_{C^{0,s+\epsilon}(\Omega_e)}+[\gamma_1^{1/2}]_{C^{0,s+\epsilon}(\Omega_e)})\|\gamma_1-\gamma_2\|_{L^{\infty}(\Omega_e)}^{1/2}\\
        &+C[\gamma_1^{1/2}-\gamma_2^{1/2}]_{C^{0,s+\epsilon}(\Omega_e)}.
    \end{split}
    \end{equation}
    By the (supercritical) Sobolev embedding in Slobodeckij spaces (cf.~\cite[Theorem~4.57]{demengel2012functional}) and the second estimate in \ref{item: assumption 3} we have
    \begin{equation}
    \label{eq: uniform bound Hölder}
        [\gamma_i^{1/2}]_{C^{0,s+\epsilon}(\Omega_e)}=[m_i]_{C^{0,s+\epsilon}(\Omega_e)}\leq C\|m_i\|_{W^{2s+\epsilon,n/s}(\Omega_e)}\leq C
    \end{equation}
    for $i=1,2$. On the other hand \cite[Theorem~4.57]{demengel2012functional}, Lemma~\ref{lem: Gagliardo Nirenberg inequality} and Lemma~\ref{lem: Holder regularity} imply
    \begin{equation}
    \label{eq: Gagliardo Nirenberg}
    \begin{split}
        [\gamma_1^{1/2}-\gamma_2^{1/2}]_{C^{0,s+\epsilon}(\Omega_e)}&\leq C\|m_1-m_2\|_{W^{2s+\epsilon,n/s}(\Omega_e)}\\
        &\leq C\|m_1-m_2\|^{\theta_0}_{W^{\frac{2s+\epsilon}{\theta_0},\theta_0 n/s}(\Omega_e)}\|m_1-m_2\|_{L^{\infty}(\Omega_e)}^{1-\theta_0}\\
        &\leq C\|m_1-m_2\|^{\theta_0}_{W^{\frac{2s+\epsilon}{\theta_0},\theta_0 n/s}(\Omega_e)}\|\gamma_1-\gamma_2\|^{\frac{1-\theta_0}{2}}_{L^{\infty}(\Omega_e)}
    \end{split}
    \end{equation}
    for all $s/n< \theta_0<1$. Note that by assumption we have $s/n<1/2$. Therefore, using the assertion \eqref{eq: difference bound} and \eqref{eq: uniform bound Hölder} we deduce from the estimate \eqref{eq: difference I1 } the following bound:
    \begin{equation}
    \label{eq: uniform bound proof}
        \|\gamma_1^{-1/2}-\gamma_2^{-1/2}\|_{C^{0,s+\epsilon}(\Omega_e)}\leq CC_{\theta_0}(\|\gamma_1-\gamma_2\|^{\frac{1}{2}}_{L^{\infty}(\Omega_e)}+\|\gamma_1-\gamma_2\|^{\frac{1-\theta_0}{2}}_{L^{\infty}(\Omega_e)}).
    \end{equation}
    Hence, we have shown
    \begin{equation}
        |I_1|\leq CC_{\theta_0}(\|\gamma_1-\gamma_2\|^{\frac{1}{2}}_{L^{\infty}(\Omega_e)}+\|\gamma_1-\gamma_2\|^{\frac{1-\theta_0}{2}}_{L^{\infty}(\Omega_e)})\|f\|_{H^s(\Omega_e)}\|g\|_{H^s(\Omega_e)}.
    \end{equation}
    Clearly the same estimate holds for $I_2$. Finally, for the expression $I_3$ we use \eqref{eq: boundedness sqrt gamma and inv sqrt gamma} to obtain
    \begin{equation}
    \label{eq: upper bound I3 }
        \begin{split}
            |I_3|&\leq C\|\Lambda_{\gamma_1}-\Lambda_{\gamma_2}\|_*\|f\|_{H^s(\Omega_e)}\|g\|_{H^s(\Omega_e)}.
        \end{split}
    \end{equation}
    Therefore, using exterior stability (cf.~Theorem~\ref{thm: exterior stability}) we have
    \[
        \begin{split}
            &\|\Lambda_{q_1}-\Lambda_{q_2}\|_*\leq CC_{\theta_0}(\|\Lambda_{\gamma_1}-\Lambda_{\gamma_2}\|_*+\|\Lambda_{\gamma_1}-\Lambda_{\gamma_2}\|_*^{\frac{1}{2}}+\|\Lambda_{\gamma_1}-\Lambda_{\gamma_2}\|_*^{\frac{1-\theta_0}{2}}).
        \end{split}
    \]
\end{proof}

\subsection{Proof of Theorem~\ref{thm: stability estimate}}

Using the reduction lemma from Section~\ref{subsec: reduction lemma}, we give here a proof of Theorem~\ref{thm: stability estimate}.
Throughout this section, we will assume without loss of generality that $\epsilon>0$ is such that $0<\epsilon\ll 1$ and $\ell s+\epsilon\not\in \N$ for $\ell=1,2$.
We split the proof into three smaller technical lemmas.
The first lemma states that under assumptions of Theorem~\ref{thm: stability estimate} the function $\widetilde{m}\vcentcolon =m/\gamma_1^{1/2}$ satisfies a fractional conductivity equation connected to the conductivities $\gamma_i$ and the difference of the potentials $q_i$. This lemma is our main tool for connecting the fractional conductivity equation to the fractional Schr\"odinger equation, which will allow us to use Theorem~\ref{thm: ruland salo proposition}, once the potentials $q_i$ are shown to be regular enough.
\begin{lemma}\label{lemma: equation for m tilde}
Let $0<s<\min(1,n/2)$, $\epsilon>0$ and assume that $\Omega\subset\R^n$ is a smooth bounded domain. Suppose that the the conductivities $\gamma_1,\gamma_2\in L^{\infty}(\R^n)$ with background deviations $m_1,m_2$ fulfill the following conditions:
    \begin{enumerate}[(i)]
        \item $\gamma_0\leq \gamma_1(x),\gamma_2(x)\leq \gamma_0^{-1}$ for some $0<\gamma_0<1$,
        \item $m_1,m_2\in  H^{4s+2\epsilon,\frac{n}{2s}}(\R^n)$
        and there exists $C_1>0$ such that
        \begin{equation}
            \|m_i\|_{H^{4s+2\epsilon,\frac{n}{2s}}(\R^n)}\leq C_1
        \end{equation}
        for $i=1,2$,
        \item $m\vcentcolon=m_1-m_2\in H^s(\R^n)$ .
    \end{enumerate}
Then there holds
\begin{align}\label{eq: eq for m tilde}
    \Div_s(\Theta_{\gamma_1}\nabla^s\widetilde{m})=\gamma_1^{1/2}\gamma_2^{1/2}(q_2-q_1)\quad \text{in}\quad\R^n,
\end{align}
where $\widetilde{m}\vcentcolon =m/\gamma_1^{1/2}$.
\end{lemma}
\begin{proof}
    First note that by assumption we have $m_i\in H^{2s,\frac{n}{2s}}(\R^n)$ for $i=1,2$ and thus we can calculate as in \cite[Proof of Lemma~8.13]{RZ2022unboundedFracCald}:
\begin{equation}
\begin{split}
  \gamma_1^{1/2}\gamma_2^{1/2} (q_2 - q_1)
  &= - \gamma_1^{1/2}\gamma_2^{1/2} \left(
  \frac{(-\Delta)^s m_2}{\gamma_2^{1/2}} - \frac{(-\Delta)^s m_1}{\gamma_1^{1/2}}\right)\\
  &= \gamma_2^{1/2} (-\Delta)^s m_1 - \gamma_1^{1/2} (-\Delta)^s m_2\\
  &=(1+m_2)(-\Delta)^s m_1 - (1+m_1) (-\Delta)^s m_2\\
  &= (1+m_2)(-\Delta)^s m_1 + (1+m_1) (-\Delta)^s m \\
  &\qquad- (1+m_1)(-\Delta)^s m_1\\
  &= \gamma_1^{1/2}(-\Delta)^s m - m(-\Delta)^s m_1.
\end{split}
\end{equation}
Setting $\widetilde{m}\vcentcolon =m/\gamma_1^{1/2}$, we obtain
\begin{equation}
\label{eq: schroedinger}
    \gamma_1^{1/2}(-\Delta)^s(\gamma_1^{1/2}\widetilde{m})+\gamma_1^{1/2}(\gamma_1^{1/2}\widetilde{m})q_1=\gamma_1^{1/2}\gamma_2^{1/2}(q_2-q_1).    
\end{equation}
Using $m_1\in H^{2s,\frac{n}{2s}}(\R^n)$ then we deduce from \cite[Corollary~A.8]{RZ2022unboundedFracCald} that there holds $\gamma_1^{1/2}\psi,\gamma_1^{-1/2}\psi\in H^s(\R^n)$ for all $\psi\in H^s(\R^n)$ and in particular $\widetilde{m}\in H^s(\R^n)$. We next observe that by the Gagliardo--Nirenberg inequality in Bessel potential spaces (cf.~\cite[Corollary~A.3,(iii)]{RZ2022unboundedFracCald}) and the monotonicity of Bessel potential spaces, we have
\begin{equation}
\label{eq: n / s estimate}
\begin{split}
     \|m_i\|_{H^{2s+\epsilon,n/s}(\R^n)}&\leq \|m_i\|^{1/2}_{H^{4s+2\epsilon,\frac{n}{2s}}(\R^n)}\|m_i\|^{1/2}_{L^{\infty}(\R^n)}
\end{split}
\end{equation}
for $i=1,2$. By the uniform ellipticity of $\gamma_i$, $i=1,2$, this immediately implies $\gamma_1^{1/2}\gamma_2^{1/2}(q_2-q_1)\in L^{n/s}(\R^n)$. By the assumptions $n/s>2$, \eqref{eq: difference bound} and the uniform ellipticity as well as interpolation in $L^p$ spaces, we see that $\gamma_1^{1/2}\gamma_2^{1/2}(q_2-q_1)\in L^2(\Omega_e)$. On the other hand, the boundedness of $\Omega$ and $n/s>2$ gives $\gamma_1^{1/2}\gamma_2^{1/2}(q_2-q_1)\in L^2(\Omega)$. Therefore, we have $\gamma_1^{1/2}\gamma_2^{1/2}(q_2-q_1)\in L^2(\R^n)$. Hence, multiplying \eqref{eq: schroedinger} by $\phi\in\schwartz(\R^n)$ and integrating over $\R^n$ shows
\[
\begin{split}
    &\int_{\R^n}(-\Delta)^s(\gamma_1^{1/2}\widetilde{m})(\gamma_1^{1/2}\phi)\,dx+\int_{\R^n}(\gamma_1^{1/2}\widetilde{m})q_1(\gamma_1^{1/2}\phi)\,dx\\
    &=\int_{\R^n}\gamma_1^{1/2}\gamma_2^{1/2}(q_2-q_1)\phi\,dx.
\end{split}
\]
Now the first integral is finite since $m\in H^{2s,\frac{n}{2s}}(\R^n)$, $\phi\in\schwartz(\R^n)$ and $\gamma_i\in L^{\infty}(\R^n)$ for $i=1,2$, the second integral by \cite[Lemma~A.10]{RZ2022unboundedFracCald} and H\"older's inequality and the integral on the right hand side by the fact that $\gamma_1^{1/2}\gamma_2^{1/2}(q_2-q_1)\in L^2(\R^n)$. Next let $(\rho_{\epsilon})_{\epsilon>0}$ be the standard mollifiers and let $m_{\epsilon}\vcentcolon =\rho_{\epsilon}\ast m$. It is well-known that $m_{\epsilon}\to m$ in $H^{2s,\frac{n}{2s}}(\R^n)$ and $H^s(\R^n)$ as $m$ satisfies $m\in H^{2s,\frac{n}{2s}}(\R^n)\cap H^s(\R^n)$. On the other hand since the Bessel potential commutes with mollification, we deduce $m_{\epsilon}\in H^t(\R^n)$ for all $t\in\R$ as $m\in L^2(\R^n)$. Therefore, we can calculate
\[
\begin{split}
    \int_{\R^n}(-\Delta)^s(\gamma_1^{1/2}\widetilde{m})(\gamma_1^{1/2}\phi)\,dx&=\lim_{\epsilon\to 0}\int_{\R^n}(-\Delta)^sm_{\epsilon}(\gamma_1^{1/2}\phi)\,dx\\
    &=\lim_{\epsilon\to 0}\int_{\R^n}(-\Delta)^{s/2}m_{\epsilon}(-\Delta)^{s/2}(\gamma_1^{1/2}\phi)\,dx\\
    &=\int_{\R^n}(-\Delta)^{s/2}m(-\Delta)^{s/2}(\gamma_1^{1/2}\phi)\,dx\\
    &= \int_{\R^n}(-\Delta)^{s/2}(\gamma^{1/2}_1\widetilde{m})(-\Delta)^{s/2}(\gamma_1^{1/2}\phi)\,dx.
\end{split}
\]
In the first equality we used the convergence $m_{\epsilon}\to m$ in $H^{2s,\frac{n}{2s}}(\R^n)$ as $\epsilon\to 0$, the continuity of the fractional Laplacian and that $\gamma_1^{1/2}\phi\in L^{\frac{n}{n-2s}}(\R^n)$, in the second equality that $m_{\epsilon}\in H^{2s}(\R^n)$, $\gamma_1^{1/2}\phi\in H^s(\R^n)$ and Plancherel's theorem, in the third equality that $m_{\epsilon}\to m$ in $H^s(\R^n)$ as $\epsilon\to 0$ and finally the definition of $\widetilde{m}$. Therefore, we obtain
\begin{equation}
\label{eq: global equation}
\begin{split}
    &\langle (-\Delta)^{s/2}(\gamma^{1/2}_1\widetilde{m}),(-\Delta)^{s/2}(\gamma_1^{1/2}\phi)\rangle_{L^2(\R^n)}+\langle q_1(\gamma_1^{1/2}\widetilde{m}),(\gamma_1^{1/2}\phi)\rangle_{L^2(\R^n)}\\
    &=\langle \gamma_1^{1/2}\gamma_2^{1/2}(q_2-q_1),\phi\rangle_{L^2(\R^n)}
\end{split}
\end{equation}
for all $\phi\in \schwartz(\R^n)$. Now, if $\phi\in H^s(\R^n)$ then we can choose a sequence $(\phi_k)_{k\in\N}\subset\schwartz (\R^n)$ such that $\phi_k\to \phi$ in $H^s(\R^n)$. By \eqref{eq: global equation} we have
\[
\begin{split}
    &\langle (-\Delta)^{s/2}(\gamma^{1/2}_1\widetilde{m}),(-\Delta)^{s/2}(\gamma_1^{1/2}\phi_k)\rangle_{L^2(\R^n)}+\langle q_1(\gamma_1^{1/2}\widetilde{m}),(\gamma_1^{1/2}\phi_k)\rangle_{L^2(\R^n)}\\
    &=\langle \gamma_1^{1/2}\gamma_2^{1/2}(q_2-q_1),\phi_k\rangle_{L^2(\R^n)}
\end{split}
\]
for all $k\in\N$. Since $\gamma_1^{1/2}\gamma_2^{1/2}(q_2-q_1)\in L^2(\R^n)$, there holds
\begin{equation}
\label{eq: convergence potential RHS}
    \langle\gamma_1^{1/2}\gamma_2^{1/2}(q_2-q_1),\phi_k\rangle_{L^2(\R^n)} \to \langle\gamma_1^{1/2}\gamma_2^{1/2}(q_2-q_1),\phi\rangle_{L^2(\R^n)}
\end{equation}
as $k\to\infty$. Again by \cite[Lemma~A.10]{RZ2022unboundedFracCald}, H\"older's inequality and the Sobolev embedding we see that
\[
    \langle q_1(\gamma_1^{1/2}\widetilde{m}),(\gamma_1^{1/2}\phi_k)\rangle_{L^2(\R^n)}  \to \langle q_1(\gamma_1^{1/2}\widetilde{m}),(\gamma_1^{1/2}\phi)\rangle_{L^2(\R^n)}
\]
as $k\to\infty$. Finally, by \cite[Corollary~A.7]{RZ2022unboundedFracCald} it follows that $\gamma_1^{1/2}\phi_k\to \gamma_1^{1/2}\phi$ in $H^s(\R^n)$ and hence $(-\Delta)^{s/2}(\gamma_1^{1/2}\phi_k)\to (-\Delta)^{s/2}(\gamma_1^{1/2}\phi)$ in $L^2(\R^n)$, but then by the Cauchy--Schwartz inequality it follows that
\[
\begin{split}
    & \langle (-\Delta)^{s/2}(\gamma^{1/2}_1\widetilde{m}),(-\Delta)^{s/2}(\gamma_1^{1/2}\phi_k)\rangle_{L^2(\R^n)}\\
    &\to \langle (-\Delta)^{s/2}(\gamma^{1/2}_1\widetilde{m}),(-\Delta)^{s/2}(\gamma_1^{1/2}\phi)\rangle_{L^2(\R^n)}
\end{split}
\]
as $k\to\infty$. Hence, \eqref{eq: global equation} holds for all $\phi \in H^s(\R^n)$.
Therefore, by the fractional Liouville reduction (Lemma~\ref{lemma: Liouville reduction}), we see that $\widetilde{m}\in H^s(\R^n)$ satisfies
\begin{align}
    \Div_s(\Theta_{\gamma_1}\nabla^s\widetilde{m})=\gamma_1^{1/2}\gamma_2^{1/2}(q_2-q_1)\quad \text{in}\quad\R^n
\end{align}
as claimed.
\end{proof}
Next we show that the assumptions of Theorem~\ref{thm: stability estimate} imply the required regularity and a priori bounds for the potentials $q_i$, allowing us to apply Theorem~\ref{thm: ruland salo proposition}.
\begin{lemma}\label{lemma: q multiplication}
Let $0<s<\min(1,n/2)$, $\epsilon>0$.
Suppose that the the conductivities $\gamma_1,\gamma_2\in L^{\infty}(\R^n)$ with background deviations $m_1,m_2$ fulfill the following conditions:
    \begin{enumerate}[(i)]
        \item\label{item: q multi assumption 1} $\gamma_0\leq \gamma_1(x),\gamma_2(x)\leq \gamma_0^{-1}$ for some $0<\gamma_0<1$,
        \item\label{item: q multi assumption 3} $m_1,m_2\in  H^{4s+2\epsilon,\frac{n}{2s}}(\R^n)$
        and there exists $C_1>0$ such that
        \begin{equation}
            \|m_i\|_{H^{4s+2\epsilon,\frac{n}{2s}}(\R^n)}\leq C_1
        \end{equation}
        for $i=1,2$,
    \end{enumerate}
Then $q_i\in H^{\delta,\frac{n}{2s}}(\R^n)$ for $\delta=2\epsilon/3$ with
\[
\|q_i \|_{H^{\delta,\frac{n}{2s}}(\R^n)} \leq M,
\]
where $M>0$ depends only on $\gamma_0$, $C_1$, $n$, $s$ and $\epsilon$.
\end{lemma}
\begin{proof}
First observe that we can write
\[
\begin{split}
    q_i&=-\frac{(-\Delta)^sm_i}{\gamma_i^{1/2}}=-(-\Delta)^sm_i\left(1-\frac{m_i}{m_i+1}\right)\\
    &=-(-\Delta)^sm_i+(-\Delta)^sm_i\frac{m_i}{m_i+1}
\end{split}
\]
for $i=1,2$. By the assumption \ref{item: q multi assumption 3} the first term belongs to $H^{2s+2\epsilon,\frac{n}{2s}}(\R^n)$ and hence it is sufficient to show that the second term is in $H^{\delta,\frac{n}{2s}}(\R^n)$ for some $\delta>0$. Now using $m_i\in H^{4s+2\epsilon,\frac{n}{2s}}(\R^n)\cap L^{\infty}(\R^n)\subset H^{2s+\epsilon,\frac{n}{s}}(\R^n)$ (see \eqref{eq: n / s estimate}), the mapping properties of the fractional Laplacian and the Sobolev embedding $H^{2s+2\epsilon,\frac{n}{2s}}(\R^n)\hookrightarrow L^{\infty}(\R^n)$ we see that $(-\Delta)^sm_i\in H^{\epsilon,\frac{n}{s}}(\R^n)\cap L^{\infty}(\R^n)$. Now we claim that $\frac{m_i}{m_i+1}\in H^{\frac{2\epsilon}{3},\frac{n}{2s}}(\R^n)\cap L^q(\R^n)$ for all $\frac{n}{2s}\leq q\leq \infty$. That $\frac{m_i}{m_i+1}\in L^q(\R^n)$ follows from the uniform ellipticity of $\gamma_i$, $m_i\in L^{\frac{n}{2s}}(\R^n)$ and interpolation in $L^p$ spaces. Next define $\Gamma_0\vcentcolon =\min(0,\gamma_0^{1/2}-1)$ and choose $\Gamma\in C^{\infty}_b(\R)$ such that $\Gamma(t)=\frac{t}{t+1}$ for $t\geq \Gamma_0$. By \cite[p.~156]{AdamsComposition} and $m_i\in H^{4s+2\epsilon,\frac{n}{2s}}(\R^n)\cap L^{\infty}(\R^n)$, we deduce for $i=1,2$ that $\Gamma(m_i)\in H^{4s+2\epsilon,\frac{n}{2s}}(\R^n)$, but since $m\geq \gamma_0^{1/2}-1>-1$ it follows that $\frac{m_i}{m_i+1}\in H^{4s+2\epsilon,\frac{n}{2s}}(\R^n)\cap L^{\infty}(\R^n)$.
Moreover, \cite[p.~156]{AdamsComposition} gives the estimate
\begin{equation}\label{eq: mi by mi+1}
\begin{split}
    \left\| \frac{m_i}{m_i+1}\right\|_{H^{4s+2\epsilon,\frac{n}{2s}}(\R^n)} &\leq C (\| m_i\|_{H^{4s+2\epsilon,\frac{n}{2s}}(\R^n)}+\alpha \|m_i\|_{H^{4s+2\epsilon,\frac{n}{2s}}(\R^n)}^{4s+2\epsilon})\\
    &\leq C,
\end{split}
\end{equation}
where $\alpha=0$ when  $4s+2\epsilon\leq 1$ and otherwise $\alpha=1$. Hence, the claim is proved. Next define
\[
    p_1\vcentcolon = \frac{n}{s},\quad p_2\vcentcolon = \frac{n}{2s},\quad r_2\vcentcolon = \frac{3n}{4s},\quad s_1\vcentcolon =\epsilon\quad \text{and}\quad \theta\vcentcolon = \frac{2}{3}.
\]
Then there holds $1<p_1,p_2,r_2<\infty$ and 
\[
    \frac{1}{p_2}=\frac{\theta}{p_1}+\frac{1}{r_2}.
\]
Moreover, since $(-\Delta)^sm_i\in H^{s_1,p_1}(\R^n)\cap L^{\infty}(\R^n)$, $\frac{m_i}{m_i+1}\in H^{\theta s_1,p_2}(\R^n)\cap L^{r_2}(\R^n)$ we deduce from \cite[Lemma~A.6]{RZ2022unboundedFracCald} that there holds
\[
    (-\Delta)^sm_i\frac{m_i}{m_i+1}\in H^{\theta s_1,p_2}(\R^n)=H^{\frac{2\epsilon}{3},\frac{n}{2s}}(\R^n).
\]
Hence, we have shown $q_i\in H^{\frac{2\epsilon}{3},\frac{n}{2s}}(\R^n)$ for $i=1,2$ as previously asserted.
Moreover, \cite[Lemma~A.6 (i)]{RZ2022unboundedFracCald} yields the estimate
\[
\begin{split}
&\left\|
(-\Delta)^sm_i\frac{m_i}{m_i+1}
\right\|_{H^{\frac{2\epsilon}{3},\frac{n}{2s}}(\R^n)}\\
&\quad\leq C
\left(
\| (-\Delta)^sm_i\|_{L^\infty(\R^n)} 
\left\|\frac{m_i}{m_i+1}\right\|_{H^{\frac{2\epsilon}{3},\frac{n}{2s}}(\R^n)}\right.
\\
&\quad\quad\left.+\left\|\frac{m_i}{m_i+1}\right\|_{L^{\frac{3n}{4s}}(\R^n)}
\| (-\Delta)^sm_i\|_{H^{\epsilon,\frac{n}{s}}(\R^n)}^\theta
\| (-\Delta)^sm_i\|_{L^\infty(\R^n)}^{1-\theta}
\right)\leq C 
\end{split}
\]
where in the second inequality we used \eqref{eq: mi by mi+1} and the assumption $m_i\in H^{4s+2\epsilon,\frac{n}{2s}}(\R^n)$.
\end{proof}

Our final technical lemma is an interpolation statement, required for the use of Theorem~\ref{thm: reduction}.
\begin{lemma}\label{lemma: W theta space for m}
Let $0<s<\min(1,n/2)$, $\epsilon>0$ and assume that $\Omega\subset\R^n$ is a smooth bounded domain.
Assume that $\theta_0\in (\max(1/2,2s/n),1)$.
Suppose that the the conductivities $\gamma_1,\gamma_2\in L^{\infty}(\R^n)$ with background deviations $m_1,m_2$ fulfill the following conditions:
    \begin{enumerate}[(i)]
        \item\label{item: interp assumption 1} $\gamma_0\leq \gamma_1(x),\gamma_2(x)\leq \gamma_0^{-1}$ for some $0<\gamma_0<1$,
        \item\label{item: interp assumption 2} $m_1,m_2\in  H^{4s+2\epsilon,\frac{n}{2s}}(\R^n)$
        and there exists $C_1>0$ such that
        \begin{equation}
            \|m_i\|_{H^{4s+2\epsilon,\frac{n}{2s}}(\R^n)}\leq C_1
        \end{equation}
        for $i=1,2$.
    \end{enumerate}
Then there holds
\begin{equation}
    \|m_i\|_{W^{\frac{2s+\epsilon}{\theta_0},\theta_0 n/s}(\Omega_e)}\leq C 
\end{equation}
for $i=1,2$ and some constant $C>0$ depending only on $n,s, \epsilon,\theta_0$ and $C_1$.
\end{lemma}
\begin{proof}
Define 
\[
    s_1\vcentcolon = 2s+\epsilon,\quad s_2\vcentcolon = 4s+2\epsilon,\quad p_1\vcentcolon = \frac{n}{s},\quad p_2\vcentcolon = \frac{n}{2s}\quad and \quad \theta\vcentcolon = 2-1/\theta_0.
\]
Since, $1/2<\theta_0<1$ we have $0<\theta<1$. Moreover, there holds $0<s_1<s_2<\infty$, $1<p_1,p_2<\infty$ and
\[
    \frac{2s+\epsilon}{\theta_0}=\theta s_1+(1-\theta)s_2\quad\text{and}\quad\frac{1}{\theta_0\frac{n}{s}}=\frac{\theta}{p_1}+\frac{1-\theta}{p_2}.
\]
Therefore, \cite[Corollary~A.3]{RZ2022unboundedFracCald} implies
\begin{equation}
\label{eq: W estimate}
    \|u\|_{H^{\frac{2s+\epsilon}{\theta_0},\theta_0\frac{n}{s}}(\R^n)}\leq C\|u\|^{\theta}_{H^{2s+\epsilon,n/s}(\R^n)}\|u\|_{H^{4s+2\epsilon,\frac{n}{2s}}(\R^n)}^{1-\theta}
\end{equation}
for all $u\in H^{2s+\epsilon,n/s}(\R^n)\cap H^{4s+2\epsilon,\frac{n}{2s}}(\R^n)$. 
By the assumptions \ref{item: interp assumption 1}-\ref{item: interp assumption 2}
and the uniform estimate \eqref{eq: n / s estimate} this ensures
\[
\begin{split}
    \|m_i\|_{H^{\frac{2s+\epsilon}{\theta_0},\theta_0\frac{n}{s}}(\R^n)}\leq C
\end{split}
\]
for $i=1,2$. On the other hand the condition $\theta_0>\frac{2s}{n}$ ensures $\theta_0\frac{n}{s}>2$ and therefore Theorem~\ref{thm: embeddings between Bessel and Gagliardo} shows
\[
    \|m_i\|_{W^{\frac{2s+\epsilon}{\theta_0},\theta_0 n/s}(\Omega_e)}\leq \|m_i\|_{W^{\frac{2s+\epsilon}{\theta_0},\theta_0 n/s}(\R^n)}\leq \|m_i\|_{H^{\frac{2s+\epsilon}{\theta_0},\theta_0 n/s}(\R^n)}\leq C
\]
for $i=1,2$.
\end{proof}

We are finally ready to complete the proof of Theorem~\ref{thm: stability estimate}:
\begin{proof}[Proof of Theorem~\ref{thm: stability estimate}]
We start by recalling that the homogeneous Sobolev space $\dot H^s(\R^n)$ is defined as the space of tempered distributions whose Fourier transform belongs to $L^1_\mathrm{loc}(\R^n)$ and satisfies
\[
\| f \|_{\dot H^s(\R^n)}^2 \vcentcolon= \int_{\R^n} |\xi|^{2s}|\widehat f(\xi)|^2 d\xi < \infty,
\]
see \cite[Section~1.3.1]{BCD-fourier-analysis-nonlinear-pde}. Since $\dot H^s(\R^n)$ continuously embeds into $L^{2n/(n-2s)}(\R^n)$ for all $0\leq s< n/2$ (see \cite[Theorem 1.38]{BCD-fourier-analysis-nonlinear-pde}), we find for any $1\leq q\leq\frac{2n}{n-2s}$
\[
\begin{split}
\|m\|_{L^q(\Omega)} &\leq C\|m\|_{L^{\frac{2n}{n-2s}}(\Omega)}
\leq C\|\widetilde{m}\|_{L^{\frac{2n}{n-2s}}(\R^n)}\\
&\leq C\|\widetilde{m}\|_{\dot{H}^s(\R^n)}
\leq C\langle \Theta_{\gamma_1}\nabla^s\widetilde{m},\nabla^s\widetilde{m}\rangle.
\end{split}
\]
Testing \eqref{eq: eq for m tilde} with $\widetilde{m}\in H^s(\R^n)$ and applying Lemma~\ref{lemma: equation for m tilde} we can estimate
    \[
    \begin{split}
    \|m\|_{L^q(\Omega)} \leq 
    C \langle \Theta_{\gamma_1}\nabla^s\widetilde{m},\nabla^s\widetilde{m}\rangle
    &\leq C\left|\int_{\R^n}\gamma_1^{1/2}\gamma_2^{1/2}(q_2-q_1)\widetilde{m}\,dx\right|\\
    &\leq C\Bigg(\left|\int_{\Omega}\gamma_1^{1/2}\gamma_2^{1/2}(q_1-q_2)\widetilde{m}\,dx\right|\\
    &\qquad+\left|\int_{\Omega_e}\gamma_1^{1/2}\gamma_2^{1/2}(q_1-q_2)\widetilde{m}\,dx\right|\Bigg)\\
    &=\vcentcolon I_1 + I_2.
    \end{split}
    \]
Observe that we have $q_i\in L^{\frac{n}{2s}}(\R^n)$ for $i=1,2$, since $m_i\in H^{2s,\frac{n}{2s}}(\R^n)$ and the conductivities are uniformly elliptic. Thus, using H\"older's inequality and then the Cauchy--Schwarz inequality, we get
\begin{equation}
\begin{split}
   I_1 &\leq C\|q_1-q_2\|_{L^{\frac{n}{2s}}(\Omega)}\|\gamma_2^{1/2}m\|_{L^{\frac{n}{n-2s}(\Omega)}}\\
   &\leq C\|q_1-q_2\|_{L^{\frac{n}{2s}}(\Omega)}\|\gamma_2^{1/2}\|_{L^{\frac{2n}{n-2s}(\Omega)}}\|m\|_{L^{\frac{2n}{n-2s}(\Omega)}}\\
   &\leq C\|q_1-q_2\|_{L^{\frac{n}{2s}}(\Omega)},
   \end{split}
\end{equation}
where in the last estimate we have used that $\Omega$ is bounded and the assumption \ref{item: main assumption 1}. On the other hand the second integral can be estimated by
\[
\begin{split}
  I_2 &\leq C\| m \gamma_2^{1/2}\|_{L^\infty(\Omega_e)}\|q_1-q_2\|_{L^1(\Omega_e)}  \\
  &\leq C\|\gamma_2^{1/2}\|_{L^\infty(\Omega_e)}\|q_1-q_2\|_{L^1(\Omega_e)} \| m \|_{L^\infty(\Omega_e)}\\
  &\leq C\frac{\|\gamma_2\|_{L^{\infty}(\Omega_e)}^{1/2}}{\gamma_0^{1/2}}\left(\|(-\Delta)^sm_1\|_{L^1(\Omega_e)}+\|(-\Delta)^sm_2\|_{L^1(\Omega_e)}\right)\|m\|_{L^{\infty}(\Omega_e)}\\
  &\leq C\|m\|_{L^{\infty}(\Omega_e)},
\end{split}
\]
where we have used the assumptions \ref{item: main assumption 1} and \ref{item: main assumption 3}. Thus, for all $1\leq q \leq 2n/(n-2s)$ there holds

\begin{equation}\label{eq: final estimate for m}
\begin{split}
        \|m\|_{L^{q}(\Omega)} &\leq C\|m\|_{L^{\frac{2n}{n-2s}}(\Omega)}\leq C(\|q_1-q_2\|_{L^{\frac{n}{2s}}(\Omega)}+\|m\|_{L^{\infty}(\Omega_e)}).    
\end{split}
\end{equation}

By Lemma~\ref{lemma: q multiplication}
the assumptions in Theorem~\ref{thm: ruland salo proposition} on the potentials $q_i$, $i=1,2$, namely that $q_i\in H^{\delta,\frac{n}{2s}}(\R^n)$ with an a priori bound $\|q_i\|_{H^{\delta,\frac{n}{2s}}(\R^n)}\leq M$, are satisfied and we can estimate the first term in the right-hand side of \eqref{eq: final estimate for m} as
\[
\| q_1 - q_2 \|_{L^{\frac{n}{2s}}(\Omega)} \leq C\omega(\| \Lambda_{q_1} - \Lambda_{q_2} \|_*),
\]
for some logarithmic modulus of continuity satisfying $\omega(x)\leq C|\log (x)|^{-\sigma}$ for all $0< x\leq 1$, where $C,\sigma>0$.
The second term of \eqref{eq: final estimate for m} can be estimated by first using Lemma~\ref{lem: Holder regularity} as
\[
\|m\|_{L^{\infty}(\Omega_e)}
= \| \gamma_1^{1/2} - \gamma_2^{1/2}\|_{L^\infty(\Omega_e)}
\leq
C \| \gamma_1 - \gamma_2\|_{L^\infty(\Omega_e)}^{1/2}
\]
and then applying the exterior stability result of Theorem~\ref{thm: exterior stability} to obtain
\[
\| \gamma_1 - \gamma_2\|_{L^\infty(\Omega_e)}^{1/2}
\leq C \| \Lambda_{\gamma_1} - \Lambda_{\gamma_2} \|_{*}^{1/2}.
\]
In summary, we have shown that there holds
\begin{equation}\label{eq: pre-stability}
\| m_1 - m_2 \|_{L^q(\Omega)} \leq C\left( \omega(\| \Lambda_{q_1} - \Lambda_{q_2} \|_*) + \| \Lambda_{\gamma_1} - \Lambda_{\gamma_2} \|_{*}^{1/2}\right).    
\end{equation}

Next, we show that $m_1,m_2$ satisfy the conditions in Theorem~\ref{thm: reduction}. Similarly, as for the estimate \eqref{eq: n / s estimate}, by the Gagliardo--Nirenberg inequality in Bessel potential spaces (cf.~\cite[Corollary~A.3,(iii)]{RZ2022unboundedFracCald}) and the monotonicity of Bessel potential spaces, we have
\[
\begin{split}
    \|m_i\|_{H^{s,n/s}(\R^n)}&\leq \|m_i\|_{H^{s+\epsilon,n/s}(\R^n)}\leq \|m_i\|^{1/2}_{H^{2s+2\epsilon,\frac{n}{2s}}(\R^n)}\|m_i\|^{1/2}_{L^{\infty}(\R^n)}\\
    &\leq \|m_i\|^{1/2}_{H^{4s+2\epsilon,\frac{n}{2s}}(\R^n)}\|m_i\|^{1/2}_{L^{\infty}(\R^n)}\leq C\|m_i\|^{1/2}_{L^{\infty}(\R^n)}
\end{split}
\]
for $i=1,2$, where in the last step we used the assumption \ref{item: main assumption 2}. Moreover, by the fact that $n/s>2$, Theorem~\ref{thm: embeddings between Bessel and Gagliardo}, (ii) and Lemma~\ref{lem: Gagliardo Nirenberg inequality}, we have
\[
\begin{split}
    \|m_i\|_{W^{2s+\epsilon,\frac{n}{s}}(\Omega_e)}&\leq \|m_i\|_{W^{2s+\epsilon,\frac{n}{s}}(\R^n)}\leq \|m_i\|_{H^{2s+\epsilon,\frac{n}{s}}(\R^n)}\\
    &\leq C\|m_i\|_{H^{4s+2\epsilon,\frac{n}{2s}}(\R^n)}^{1/2}\|m_i\|_{L^{\infty}(\R^n)}^{1/2}
\end{split}
\]
for $i=1,2$. Now, using the uniform bound \eqref{eq: main bound 1} and the uniform ellipticity of $\gamma_i$ we deduce
\[
    \|m\|_{W^{2s+\epsilon,n/s}(\Omega_e)}\leq C.
\]
Applying Lemma~\ref{lemma: W theta space for m} we also see that
\begin{equation}
\label{eq: uniform control}
    \|m_i\|_{W^{\frac{2s+\epsilon}{\theta_0},\theta_0 n/s}(\Omega_e)}\leq C .
\end{equation}
This now demonstrates that $m_1-m_2$ satisfies the condition \eqref{eq: difference bound} in Theorem~\ref{thm: reduction}. Therefore, we can apply Theorem~\ref{thm: reduction} to deduce the estimate 
\[
\begin{split}
 \| m_1 - m_2 \|_{L^q(\Omega)} 
        &\leq C\left( \omega\Big( \|\Lambda_{\gamma_1} - \Lambda_{\gamma_2}\|_*
        +\|\Lambda_{\gamma_1} - \Lambda_{\gamma_2} \|_*^\frac{1}{2}\right.\\
        &\qquad\left.+\|\Lambda_{\gamma_1} - \Lambda_{\gamma_2} \|_*^{\frac{1-\theta_0}{2}}
        \Big) + \| \Lambda_{\gamma_1} - \Lambda_{\gamma_2}     \|_{*}^{\frac{1}{2}}\right)
\end{split}
\]
for all $1\leq q\leq \frac{2n}{n-2s}$. Since $\theta_0\in (1/2,1)$ we have $\frac{1-\theta_0}{2}\in (0,1/4)$. Hence, there holds $x\leq x^{1/2}\leq x^{\frac{1-\theta_0}{2}}$ for all $0<x\leq 1$. Therefore, we obtain
\[
\begin{split}
&\omega\left(\|\Lambda_{\gamma_1} - \Lambda_{\gamma_2}\|_*
        +\|\Lambda_{\gamma_1} - \Lambda_{\gamma_2} \|_*^\frac{1}{2}+\|\Lambda_{\gamma_1} - \Lambda_{\gamma_2} \|_*^{\frac{1-\theta_0}{2}}\right)\\
        &\leq \omega\left(3\|\Lambda_{\gamma_1} - \Lambda_{\gamma_2}\|_*^{\frac{1-\theta_0}{2}}\right).    
\end{split}
\]
By the assumptions $\|\Lambda_{\gamma_1}-\Lambda_{\gamma_2}\|_*\leq 3^{-1/\delta}$, $0<\delta<\frac{1-\theta_0}{2}$, we have
\[
3\|\Lambda_{\gamma_1} - \Lambda_{\gamma_2}\|_*^{\frac{1-\theta_0}{2}}
\leq
\|\Lambda_{\gamma_1} - \Lambda_{\gamma_2}\|_*^{{\frac{1-\theta_0}{2}}-\delta}.
\]
Using the fact that $\omega(x)\leq C|\log x|^{-\sigma}$ for $0<x\leq 1$, we deduce
\[
\begin{split}
&\omega\left(\|\Lambda_{\gamma_1} - \Lambda_{\gamma_2}\|_*
        +\|\Lambda_{\gamma_1} - \Lambda_{\gamma_2} \|_*^\frac{1}{2}+\|\Lambda_{\gamma_1} - \Lambda_{\gamma_2} \|_*^{\frac{1-\theta_0}{2}}\right)\\
        &\leq C\omega\left(\|\Lambda_{\gamma_1} - \Lambda_{\gamma_2}\|_*\right).    
\end{split}
\]
Next we observe that there holds $x^{1/2}\leq C|\log x|^{-\sigma}$ for all $0<x\leq 1$ and some $C>0$. To see this, observe that this estimate is equivalent to $Cx^{-1/2}\geq 2^{-\sigma}|\log x^{-1/2}|^{\sigma}$. Since $0<x\leq 1$, this is the same as $\log x^{-1/2}\leq \frac{C^{1/\sigma}}{2}(x^{-1/2})^{1/\sigma}$, but it is well-known that there holds $\log(y)\leq y^r/r$ for all $y>0$ and $r>0$. In fact, the last assertion is a straightforward consequence of the inequality $\log(z)\leq z-1$ for all $z>0$ by applying it to $z=y^r$ with $y>0$, $r>0$.  Hence, the above estimate holds with $C=(2\sigma)^{\sigma}$. This finally shows
\[
\|\Lambda_{\gamma_1} - \Lambda_{\gamma_2}\|_* \leq C |\log(\|\Lambda_{\gamma_1} - \Lambda_{\gamma_2}\|_*)|^{-\sigma}=C\omega(\|\Lambda_{\gamma_1}-\Lambda_{\gamma_2}\|_*),
\]
and we can conclude the proof.
\end{proof}

\section{Partial data reduction with minimal regularity assumptions on the domain}\label{sec: partial data reduction}

For the sake of completeness and possible future work on low regularity settings, we record a proposition considering a quantitative partial data reduction to the Schrödinger case with minimal assumptions on the domain $\Omega$. These improvements come with the cost of having certain restrictions on the considered sets of measurements but on the other hand also allow a simpler proof. Furthermore, Proposition \ref{thm: partial data reduction} is strong enough to conclude Theorem \ref{thm: stability estimate} after minor changes to the proof and using the partial data stability result for the Schrödinger case in \cite{RS-fractional-calderon-low-regularity-stability}. This argument however does not establish the quantitative reduction ''up to the boundary'' like the proof of Theorem \ref{thm: reduction}. In particular, this approach avoids using $W^{s,p}$ spaces and the explicit extension operators, which in part explains why the boundary regularity questions are not encountered and the proof is considerably simpler. 

Finally, we emphasize that the partial data reduction to the Schrödinger case does not directly imply partial data stability for the conductivity case as the partial data uniqueness result is based on an additional unique continuation argument for the conductivities (see \cite{RGZ2022GlobalUniqueness,RZ2022LowReg}) and the authors are not aware of quantitative unique continuation results of the following type: Let $1 \leq p,q, r \leq \infty$, $s,t \geq 0$ and $W \subset \R^n$ be a nonempty open set. For all $u \in X \subset H^{t,q}(\R^n)$ there holds \begin{equation}\label{eq: qUCP}\norm{u}_{L^p(\R^n)} \leq F(\norm{(-\Delta)^su}_{L^r(W)},\norm{u}_{L^\infty(W)})\end{equation} for some continuous function $F: [0,\infty) \times [0,\infty) \to [0,\infty)$ with $F(0,0)=0$ and independent of $u \in X$ where the set $X$ encodes possible a priori assumptions. For the relevant regularity assumptions and choices of $p,q,r,s,t$, see \cite{RGZ2022GlobalUniqueness,RZ2022LowReg} where $u=m_1-m_2$ is the choice of $u$ in our possible application. Such estimates for the special case $u|_W=0$, i.e. $\norm{u}_{L^\infty(W)}=0$, could already give new results towards the stability of the partial data problems for the fractional conductivity equation. In general, estimates of the type \eqref{eq: qUCP} would be interesting also in other function spaces and norms.

\begin{proposition}
\label{thm: partial data reduction}
    Let $0<s<\min(1,n/2)$, $s/n < \theta_0 < 1$ and $\epsilon > 0$.
    Let $\Omega\subset \R^n$ be a nonempty open set bounded in one direction. Suppose that the conductivities $\gamma_1,\gamma_2\in L^{\infty}(\R^n)$ with background deviations $m_1,m_2$ and potentials $q_1,q_2$ fulfill the following conditions:
    \begin{enumerate}[(i)]
        \item\label{item: partial assumption 1} $\gamma_0\leq \gamma_1(x),\gamma_2(x)\leq \gamma_0^{-1}$ for some $0<\gamma_0<1$,
        \item\label{item: partial assumption 3} 
       There exists $C_0>0$ such that
        \begin{equation}
        \label{eq: partial difference bound}
            \|m_i\|_{H^{\frac{2s+\epsilon}{\theta_0},\theta_0 n/s}(\R^n)}\leq C_0.
        \end{equation}
         for $i=1,2$.
    \end{enumerate} 
    Let $W_1,W_2,W \Subset \Omega_e$ be nonempty open sets such that $W_1 \cup W_2 \Subset W$. Then there holds
    \begin{equation}
    \label{eq: partial reduction}
    \begin{split}
        &\|\Lambda_{q_1}-\Lambda_{q_2}\|_{\widetilde{H}^s(W_1) \to (\widetilde{H}^s(W_2))^*} \\
        &\quad\leq C(\|\Lambda_{\gamma_1}-\Lambda_{\gamma_2}\|_{\widetilde{H}^s(W) \to (\widetilde{H}^{s}(W))^*}+\|\Lambda_{\gamma_1}-\Lambda_{\gamma_2}\|_{\widetilde{H}^s(W) \to (\widetilde{H}^{s}(W))^*}^{\frac{1-\theta_0}{2}})
    \end{split}
    \end{equation}
where $C>0$ depending only on $s,\epsilon,n,\Omega, C_0, \theta_0, W_1, W_2, W$ and $\gamma_0$.
\end{proposition}
\begin{proof} Suppose that $f \in C_c^\infty(W_1)$ and $g \in C_c^\infty(W_2)$. Choose a smooth cutoff function $\eta|_{W_1\cup W_2}=1$, $0 \leq \eta \leq 1$ and $\supp(\eta) \subset W$. We explain next how to modify the proof of Theorem \ref{thm: reduction}, in order to obtain the quantitative partial data reduction result. To do so, we will establish sufficient estimates next. We first note that $m_1,m_2 \in H^{2s+\epsilon,n/s}(\R^n) \supset H^{s,n/s}(\R^n)$ with explicit bounds for the norms by \ref{item: partial assumption 1}, \ref{item: partial assumption 3}, the Gagliardo--Nirenberg inequality in Bessel potential spaces and the monotonicity of Bessel potential spaces. Therefore we may continue the proof of Theorem \ref{thm: reduction} up to the point where we have the terms $I_1, I_2, I_3$ as in the proof of Theorem \ref{thm: reduction}.

Since $W_1,W_2 \subset W$, we have
\begin{equation}
        |I_3|\leq C\|\Lambda_{\gamma_1}-\Lambda_{\gamma_2}\|_{\widetilde{H}^s(W) \to (\widetilde{H}^{s}(W))^*}\|f\|_{H^s(\R^n)}\|g\|_{H^s(\R^n)}.
\end{equation}
as in the earlier proof.

We may suppose, by taking $\epsilon$ smaller if necessary, that $0 <\epsilon <1-s$ and $2s+\epsilon$ is not an integer by the monotonicity of Bessel potential spaces. For the term $I_1$, we may calculate that
    \begin{equation}
    \label{eq: partial estimate I1}
        \begin{split}
            |I_1|&\leq \|\Lambda_{\gamma_1}\|_{\widetilde{H}^s(W) \to (\widetilde{H}^{s}(W))^*}\|\gamma_1^{-1/2}f\|_{H^s(\R^n)}\|(\gamma_1^{-1/2}-\gamma_2^{-1/2})g\|_{H^s(\R^n)}\\
            &\leq C\|\Lambda_{\gamma_1}\|_{\widetilde{H}^s(W) \to (\widetilde{H}^{s}(W))^*}\|f\|_{H^s(\R^n)}\|\eta(\gamma_1^{-1/2}-\gamma_2^{-1/2})\|_{C^{0,s+\epsilon}(\R^n)}\|g\|_{H^s(\R^n)}
        \end{split}
    \end{equation}
since $g=\eta g$.
We may then estimate using the embeddings to Hölder spaces, Gagliardo--Nirenberg inequality in Bessel potential spaces, the formula \eqref{eq: L infty bound}, boundedness of the multiplication with $\eta$ and support conditions, that
\begin{equation}
\begin{split}
    &\|\eta(\gamma_1^{-1/2}-\gamma_2^{-1/2})\|_{C^{0,s+\epsilon}(\R^n)} \\
    &\quad\leq C\|\eta(\gamma_1^{-1/2}-\gamma_2^{-1/2})\|_{H^{2s+\epsilon,n/s}(\R^n)}\\
        &\quad\leq C\|\eta(\gamma_1^{-1/2}-\gamma_2^{-1/2})\|^{\theta_0}_{H^{\frac{2s+\epsilon}{\theta_0},\theta_0 n/s}(\R^n)}\|\eta(\gamma_1^{-1/2}-\gamma_2^{-1/2})\|_{L^{\infty}(\R^n)}^{1-\theta_0}\\
        &\quad\leq C\|\gamma_1^{-1/2}-\gamma_2^{-1/2}\|^{\theta_0}_{H^{\frac{2s+\epsilon}{\theta_0},\theta_0 n/s}(\R^n)}\|\gamma_1^{-1/2}-\gamma_2^{-1/2}\|_{L^{\infty}(W)}^{1-\theta_0}\\
        &\quad\leq C\left\|\frac{m_1}{m_1+1}-\frac{m_2}{m_2+1}\right\|^{\theta_0}_{H^{\frac{2s+\epsilon}{\theta_0},\theta_0 n/s}(\R^n)}\|\gamma_1-\gamma_2\|^{\frac{1-\theta_0}{2}}_{L^{\infty}(W)} \\
        &\quad\leq C\|\gamma_1-\gamma_2\|^{\frac{1-\theta_0}{2}}_{L^{\infty}(W)} \\
\end{split}
\end{equation}
where in the last step we used the triangle inequality and a composition estimate for functions in Bessel potential spaces (see e.g.~\cite{AdamsComposition} and references therein). In fact, for any $t >0 $, $1 < p < \infty$, $\delta >0$, there exists a polynomial function $P\colon \R^2 \to \R$ (of degree at most $t+1$ and with nonnegative coefficients) such that
\begin{equation}
    \left\|\frac{m}{m+1}\right\|_{H^{t,p}(\R^n)} \leq P(\norm{m}_{H^{t,p}(\R^n)},\norm{m}_{L^\infty(\R^n)}) 
\end{equation}
for all $m \in H^{t,p}(\R^n)\cap L^\infty(\R^n)$ with $m+1 \geq \delta$. This can be argued similarly as \eqref{eq: mi by mi+1} in the proof of Lemma \ref{lemma: q multiplication} but we decided to recall this alternative estimate here.
This let us conclude that
\begin{equation}
        |I_1|\leq C\|f\|_{H^s(\R^n)}\|g\|_{H^s(\R^n)}\|\Lambda_{\gamma_1}-\Lambda_{\gamma_2}\|_{\widetilde{H}^s(W) \to (\widetilde{H}^{s}(W))^*}^{\frac{1-\theta_0}{2}}
\end{equation}
by the exterior stability estimate (Theorem \ref{eq: exterior stability}) since $\gamma_1,\gamma_2$ are continuous by \ref{item: partial assumption 3}. This is possible since the exterior stability estimate also holds for the considered partial data, i.e. $\norm{\gamma_1-\gamma_2}_{L^\infty(W)} \leq C\|\Lambda_{\gamma_1}-\Lambda_{\gamma_2}\|_{\widetilde{H}^s(W) \to (\widetilde{H}^{s}(W))^*}$. We may argue similarly with $I_2$, which completes the proof.
\end{proof}

\begin{remark} Theorem \ref{thm: reduction} could be also adapted into the partial data setting as in Proposition \ref{thm: partial data reduction}. We omit presenting these details.
\end{remark}

\section{Exponential instability}\label{sec: instability}

In this Section we complement the above considerations about stability with an instability result in the flavour of \cite{KochRulandSaloInstability}. We start by recalling the definitions of $\epsilon$-discrete sets and $\delta$-nets. These can be given in the setting of a generic metric space $(X,d)$.

\begin{definition}
Let $(X,d)$ be a metric space, and assume $\epsilon,\delta>0$. A set $Y\subset X$ is said to be $\epsilon$-discrete if for all $y_1,y_2\in Y$ with $y_1\neq y_2$ it holds $d(y_1,y_2)\geq\epsilon$. A set $Z\subset X$ is said to be a $\delta$-net for a set $X_1\subset X$ if for all $x\in X_1$ there exists $z\in Z$ such that $d(x,z)\leq \delta$.
\end{definition}

We can now prove Theorem \ref{thm: instability}, which shows that the exponential stability obtained in Section \ref{sec: stability estimates} can not be improved. For this result it will suffice to consider conductivities whose exterior value is the constant $1$.

\begin{proof}[Proof of Theorem \ref{thm: instability}]
Define the set
$$ X_{\ell\epsilon\beta} \vcentcolon = \{f\in C^\ell_c(B_1) \,;\, \|f\|_{L^\infty}\leq \epsilon, \|f\|_{C^\ell}\leq\beta\}, $$
and let $\widetilde X_{\ell\epsilon\beta} \vcentcolon = 1+X_{\ell\epsilon\beta}$. By \cite[Lemma~2]{MandacheInstability} (see also \cite[Lemma~3.4]{RS-Instability}) we deduce the existence of an $\epsilon$-discrete set $\widetilde Z\subset \widetilde X_{\ell\epsilon\beta}$ of cardinality $$|\widetilde Z|\geq \exp \left( C(\beta/\epsilon)^{n/\ell} \right),$$ with $C=C(\ell,n)>0$, where $\widetilde X_{\ell\epsilon\beta}$ is seen as a metric space with respect to the $L^\infty$ norm. By careful construction, it is also possible to ensure that $1\leq \gamma \leq 2$ for all $\gamma\in\widetilde Z$ (see \cite[Proof of Corollary 1]{MandacheInstability}). Let now $\gamma\in\widetilde Z$, and let $q$ be the corresponding transformed potential. Since $\gamma\geq 1$, for all $v\in\tilde H^s(B_1)$ we get
\begin{align*}
    \langle \mbox{div}_s\Theta_\gamma\nabla^s v,v \rangle_{H^{-s}(\R^n)\times H^s(\R^n)} & =B_{\gamma}(v,v) \\ & \geq \|(-\Delta)^{s/2}v\|_{L^2(\R^n)}^2\\ & \geq \lambda_{1,s}\|v\|^2_{L^2(B_1)},
\end{align*} where $\lambda_{1,s}$ is the first Dirichlet eigenvalue of $(-\Delta)^s$ in $B_1$ (see e.g.~\cite[Proof of Lemma~3.2]{RS-Instability}). Therefore, by the fractional Liouville reduction
\begin{align*}
    \langle ((-\Delta)^s+q)v,v \rangle_{H^{-s}(\R^n)\times H^s(\R^n)} & = \langle \mbox{div}_s\Theta_\gamma\nabla^s(\gamma^{-1/2}v),\gamma^{-1/2}v \rangle_{H^{-s}(\R^n)\times H^s(\R^n)} \\ &  \geq \lambda_{1,s}\|\gamma^{-1/2}v\|^2_{L^2(B_1)} \\ & \geq \frac{\lambda_{1,s}}{2}\|v\|^2_{L^2(B_1)},
\end{align*}
and eventually $\|((-\Delta)^s+q)^{-1}\|_{L^2(B_1)\rightarrow L^2(B_1)} \leq \frac{2}{\lambda_{1,s}}$.
\\

Since $m\vcentcolon =\gamma^{1/2}-1\geq 0$, for the potential we compute
\begin{align*}
    \|q\|_{L^\infty(B_1)} & = \left\| \frac{(-\Delta)^sm}{1+m} \right\|_{L^\infty(B_1)} \lesssim \left\| (-\Delta)^sm\right\|_{L^\infty(B_1)} \leq \left\| (-\Delta)^sm\right\|_{C^{\ell-2s}(\R^n)}.
\end{align*}
 Observe now that the fractional Laplacian $(-\Delta)^s$ acts as a bounded operator between $C^r$ and $C^{r-2s}$ for any $r,r-2s\in \mathbb R^+\setminus \N$. In order to see this, write the symbol as $$|\xi|^{2s}=\psi(\xi)|\xi|^{2s}+(1-\psi(\xi))|\xi|^{2s},$$ where $\psi\in C^\infty_c(\R^n)$ is $1$ near the origin. The second term on the right hand side belongs to H\"ormander's class $S^{2s}_{1,0}$, and thus has the correct mapping properties (see e.g.~\cite{Taylor96}). The first term on the right hand side corresponds to a convolution operator with kernel $k=\mathcal F^{-1}(\psi)\ast \mathcal F^{-1}(|\xi|^{2s})$, which is $L^1$ as a convolution of a Schwartz function with a homogeneous function of order $-n-2s$. Therefore $u\mapsto k\ast u$ is bounded between any two H\"older spaces by the Fourier characterization of H\"older spaces (\cite[Section 2.3.7]{Triebel-Theory-of-function-spaces}). 
 Thus
 $$ \|q\|_{L^\infty(B_1)} \lesssim \left\| (-\Delta)^sm\right\|_{C^{\ell-2s}(\R^n)} \lesssim \left\| m\right\|_{C^{\ell}(\R^n)}. $$
 Moreover, $\gamma\in \widetilde Z \subset \widetilde X_{\ell\epsilon\beta}$ implies $$m^2+2m=(m+1)^2-1=\gamma-1\in X_{\ell\epsilon\beta},$$ and thus in particular $m^2+2m\in C^\ell_c(B_1)$. Define the function $F \colon x\mapsto \sqrt{1+x}-1$, which is smooth for non-negative $x$ and has the property that $F(m^2+2m)=m$. Since $C^\ell_c$ is closed under composition with smooth functions, we have $m\in C^\ell_c(B_1)$ as well, with $\|m\|_{C^\ell(B_1)}\lesssim \|m^2+2m\|_{C^\ell(B_1)} $.
 
 Eventually
$$\|q\|_{L^\infty(B_1)}\lesssim \left\| m\right\|_{C^{\ell}(\R^n)} \lesssim \|\gamma-1\|_{C^\ell(B_1)}\leq \beta.$$

We shall now apply \cite[Proposition~2.4]{RS-Instability}. Observe that the eigenvalue condition of the said proposition is not needed here, because in this case the potential $q$ is comes from a fractional Liouville reduction, and so the Dirichlet problem for the transformed operator is already known to be well-posed. 

Recall that $\{f_{h,k,l}\}$ is the basis of $L^2(B_3\setminus\overline B_2)$ constructed in \cite[Lemma~2.1]{RS-Instability}. The operator $\Gamma(q)\vcentcolon = \Lambda_q-\Lambda_0$ mapping $L^2(B_3\setminus\overline B_2)$ to itself is completely characterized by the quantities
$$a^{h_2,k_2,l_2}_{h_1,k_1,l_1}(q)\vcentcolon = \langle \Gamma(q) f_{h_1,k_1,l_1},f_{h_2,k_2,l_2}\rangle_{L^2(B_3\setminus\overline B_2)}.$$ 
Let
$$X\vcentcolon = \{\,\Gamma(q) \,;\, q\in \mathcal Q,  \|\Gamma(q)\|_X<\infty\,\},$$
where $\mathcal{Q}$ is the class of all transformed potentials, and $$\|\Gamma(q)\|_X\vcentcolon = \sup_{h_i,k_i,l_i}(1+\max\{h_1+k_1,h_2+k_2\})^{n+2}|a^{h_2,k_2,l_2}_{h_1,k_1,l_1}(q)|.$$
By \cite[Proposition~2.4]{RS-Instability} we obtain
\begin{align*} |a^{h_2,k_2,l_2}_{h_1,k_1,l_1}(q)| & \leq C_{n,s}e^{-c\max\{ h_1+k_1,h_2+k_2 \}}\|q\|_{L^\infty(B_1)}\|((-\Delta)^s+q)^{-1}\|_{L^2(B_1)\rightarrow L^2(B_1)} \\ & \leq C'_{n,s} \beta e^{-c\max\{ h_1+k_1,h_2+k_2 \}},
\end{align*}
which means that if $\gamma\in\widetilde X_{\ell\epsilon\beta}$ then $\Gamma(q)\in X$.

With this in mind, we can follow the proof of Lemma 3 in \cite{MandacheInstability} (see also \cite[Lemma~3.2]{RS-Instability}) to construct a $\delta$-net $\overline Y$ for the image under $\Gamma \colon q\mapsto \Lambda_q-\Lambda_0$ of the set of potentials corresponding to conductivities in $\widetilde X_{\ell\epsilon\beta}$. Here $\delta \vcentcolon = \exp \left(-\epsilon^{-\frac{n}{(2n+3)\ell}}\right)$, and the cardinality of $\overline Y$ is
$$ |\overline Y| \leq \beta \exp \left( C'' \epsilon^{-n/\ell}\right). $$
It is clear that for $\beta$ large enough it must hold that $|\widetilde Z| > |\overline Y|$, which means that there exists two conductivities $\gamma_1,\gamma_2 \in \widetilde X_{\ell\epsilon\beta}$ with $\|\gamma_1-\gamma_2\|_{L^\infty(B_1)}\geq \epsilon$ and $$\|\Lambda_{q_1}-\Lambda_{q_2}\|_{L^2(B_3\setminus\overline B_2)\rightarrow L^2(B_3\setminus\overline B_2)} \lesssim \|\Gamma(q_1)-\Gamma(q_2)\|_X \lesssim \delta,$$
in light of the fact that $\Lambda_q$ is a bounded operator $L^2(B_3\setminus\overline{B}_2)\to L^2(B_3\setminus\overline{B}_2)$ (see \cite[Remarks~2.2, 2.5]{RS-Instability}) and the related estimate \cite[eq. (21)]{RS-Instability}.\\

Since $\gamma_1=\gamma_2=1$ in $\R^n\setminus \overline B_1$, by \cite[Lemma~4.1]{RGZ2022GlobalUniqueness} we deduce $\Lambda_q = \Lambda_\gamma$ as operators on $H^s(\R^n\setminus\overline{B}_1)\to (H^s(\R^n\setminus\overline{B}_1))^*$. This lets us conclude that
\begin{align*}
\|\Lambda_{\gamma_1}-\Lambda_{\gamma_2}\|_{H^s(B_3\setminus\overline{B}_2)\to (H^s(B_3\setminus\overline{B}_2))^*}& \leq \|\Lambda_{q_1}-\Lambda_{q_2}\|_{H^s(B_3\setminus\overline{B}_2)\to (H^s(B_3\setminus\overline{B}_2))^*} \\ & \leq \|\Lambda_{q_1}-\Lambda_{q_2}\|_{L^2(B_3\setminus\overline B_2)\rightarrow L^2(B_3\setminus\overline B_2)} \\ & \lesssim  \delta. 
\end{align*} 
\end{proof}

\appendix

\section{Proofs of auxiliary results}
\label{sec: Proofs of auxiliary lemmas}

\begin{proof}[Proof of Theorem~\ref{thm: embeddings between Bessel and Gagliardo}]
    Let us first consider the case $\Omega=\R^n$. In fact, this follows from embeddings between different function spaces:
    \begin{enumerate}[(i)]
        \item\label{Triebel Lizorkin monotone} If $s\in\R$, $0<p<\infty$ and $0<q_0\leq q_1\leq \infty$ then $F^s_{p,q_0}(\R^n)\hookrightarrow F^s_{p,q_1}(\R^n)$ (cf.~\cite[Section~2.3.2, Proposition~ 2]{Triebel-Theory-of-function-spaces}).
        \item\label{Triebel Lizorkin equal Bessel} If $s\in\R$, $1<p<\infty$ then $F^s_{p,2}(\R^n)=H^{s,p}(\R^n)$ (cf.~\cite[Section~2.3.5, eq. (2)]{Triebel-Theory-of-function-spaces}).
        \item\label{Triebel Lizorkin equal Besov} If $s\in\R$, $0<p<\infty$ then $F^s_{p,p}(\R^n)= B^s_{p,p}(\R^n)$ (cf.~\cite[Section~2.3.2, Proposition~ 2]{Triebel-Theory-of-function-spaces}).
        \item\label{Besov equal Lipschitz} If $s>0$, $1\leq p<\infty$ and $1\leq q\leq \infty$ then $B^s_{p,q}(\R^n)=\Lambda^s_{p,q}(\R^n)$ (cf.~\cite[Section~2.3.5, eq. (3)]{Triebel-Theory-of-function-spaces}).
        \item\label{Lipschitz equal Slobodeckij} If $s\in\R_+\setminus\N$, $1\leq p<\infty$ then $\Lambda^s_{p,p}(\R^n)=W^{s,p}(\R^n)$ (cf.~\cite[Section~2.2.2, Remark~3]{Triebel-Theory-of-function-spaces}).
    \end{enumerate}
    Above $F^s_{p,q}(\R^n)$ denotes the Triebel--Lizorkin spaces, $B^s_{p,q}(\R^n)$ the Besov spaces and $\Lambda^s_{p,q}(\R^n)$ the Lipschitz spaces. For their definition and more details we refer to the monograph \cite{Triebel-Theory-of-function-spaces}. These embeddings imply
    \begin{equation}
    \label{eq: embedding p<2}
        W^{s,p}(\R^n)\overset{\ref{Lipschitz equal Slobodeckij}}{=}\Lambda^s_{p,p}(\R^n)\overset{\ref{Besov equal Lipschitz}}{=}B^s_{p,p}(\R^n)\overset{\ref{Triebel Lizorkin equal Besov}}{=}F^{s}_{p,p}(\R^n)\overset{\ref{Triebel Lizorkin monotone}}{\hookrightarrow}F^s_{p,2}(\R^n)\overset{\ref{Triebel Lizorkin equal Bessel}}{=}H^{s,p}(\R^n)
    \end{equation}
    if $s\in\R_+\setminus\N$ and $1<p\leq 2$. On the other hand if $s\in\R_+$, $2\leq p<\infty$ then we have
    \begin{equation}
    \label{eq: embedding p>2}
        H^{s,p}(\R^n)\overset{\ref{Triebel Lizorkin equal Bessel}}{=}F^s_{p,2}(\R^n)\overset{\ref{Triebel Lizorkin monotone}}{\hookrightarrow}F^{s}_{p,p}(\R^n)\overset{\ref{Triebel Lizorkin equal Besov}}{=}B^s_{p,p}(\R^n)\overset{\ref{Besov equal Lipschitz}}{=}\Lambda^s_{p,p}(\R^n)\overset{\ref{Lipschitz equal Slobodeckij}}{=}W^{s,p}(\R^n).
    \end{equation}
    Now assume that $2\leq p<\infty$ and $\Omega\subset\R^n$ is an arbitrary open set. Let $u\in H^{s,p}(\Omega)$ and assume $v\in H^{s,p}(\R^n)$ satisfies $u=v|_{\Omega}$. Then by \eqref{eq: embedding p>2} there holds $v\in W^{s,p}(\R^n)$ with
    \[
        \|v\|_{W^{s,p}(\R^n)}\leq C\|v\|_{H^{s,p}(\R^n)}
    \]
    for some $C>0$. By defintion of the Slobodeckij spaces and $v|_{\Omega}=u$ we have
    \[
        \|u\|_{W^{s,p}(\Omega)}\leq \|v\|_{W^{s,p}(\R^n)}\leq C\|v\|_{H^{s,p}(\R^n)}.
    \]
    Taking the infimum over all extensions of $u$ and recalling the definition of the norm $\|\cdot\|_{H^{s,p}(\Omega)}$ we deduce
    \[
        \|u\|_{W^{s,p}(\Omega)}\leq C\|u\|_{H^{s,p}(\Omega)}.
    \]
    This proves Theorem~\ref{thm: embeddings between Bessel and Gagliardo}, \ref{item: embedding p > 2}. Next let $1<p\leq 2$, $s=k+\sigma$ with $k\in\N$, $0<\sigma<1$ and assume $\Omega$ is a $C^{k,1}$ domain with bounded boundary. By Lemma~\ref{lem: extension} there is an extension operator $E\colon W^{s,p}(\Omega)\to W^{s,p}(\R^n)$ such that $Eu|_{\Omega}=u$ and $\|Eu\|_{W^{s,p}(\R^n)}\leq C\|u\|_{W^{s,p}(\Omega)}$ for all $u\in W^{s,p}(\Omega)$ and some $C>0$. Using \eqref{eq: embedding p>2} we have
    \[
        \|Eu\|_{H^{s,p}(\R^n)}\leq C\|Eu\|_{W^{s,p}(\R^n)}\leq C\|u\|_{W^{s,p}(\Omega)}.
    \]
    By the definition of the $\|\cdot\|_{H^{s,p}(\Omega)}$ norm this implies
    \[
        \|u\|_{H^{s,p}(\Omega)}\leq C\|u\|_{W^{s,p}(\Omega)}
    \]
    and we can conclude the proof of Theorem~\ref{thm: embeddings between Bessel and Gagliardo}, \ref{item: embedding p < 2}.
\end{proof}

\begin{proof}[Proof of Lemma \ref{lem: multiplication Hoelder}]
    Statement \ref{item: basic case} has been essentially proved in \cite[Lemma~5.3]{DINEPV-hitchhiker-sobolev} but under the assumptions $0\leq \phi\leq 1$ and $\mu=1$. We give a proof of this slightly more general result here for the sake of completeness. Since $C^{0,\mu}(\Omega)\subset L^{\infty}(\Omega)$, we have $\|\phi u\|_{L^p(\Omega)}\leq \|\phi\|_{L^{\infty}(\Omega)}\|u\|_{L^p(\Omega)}$ and hence it remains to control the Gagliardo seminorm. We have
    \[
    \begin{split}
        &[\phi u]_{W^{s,p}(\Omega)}^p=\int_{\Omega}\int_{\Omega}\frac{|(\phi u)(x)-(\phi u)(y)|^p}{|x-y|^{n+sp}}\,dxdy\\
        &= \int_{\Omega}\int_{\Omega}\frac{|\phi(x)(u(x)-u(y))+(\phi(x)-\phi(y))u(y)|^p}{|x-y|^{n+sp}}\,dxdy\\
        &\leq 2^{p-1}\left(\int_{\Omega}\int_{\Omega}\frac{|\phi(x)|^p|u(x)-u(y)|^p}{|x-y|^{n+sp}}\,dxdy+\int_{\Omega}\int_{\Omega}\frac{|\phi(x)-\phi(y)|^p|u(y)|^p}{|x-y|^{n+sp}}\,dxdy\right)\\
        &\leq 2^{p-1}\left(\|\phi\|^p_{L^{\infty}(\Omega)}[u]_{W^{s,p}(\Omega)}^p+\int_{\Omega}\int_{\Omega}\frac{|\phi(x)-\phi(y)|^p|u(y)|^p}{|x-y|^{n+sp}}\,dxdy\right).
    \end{split}
    \]
    Now we write
    \[
    \begin{split}
        &\int_{\Omega}\int_{\Omega}\frac{|\phi(x)-\phi(y)|^p|u(y)|^p}{|x-y|^{n+sp}}\,dxdy=\int_{\Omega}\int_{\Omega}\frac{|\phi(x)-\phi(y)|^p|u(y)|^p}{|x-y|^{n+sp}}\chi_{B_1(y)}(x)\,dxdy\\
        &+\int_{\Omega}\int_{\Omega}\frac{|\phi(x)-\phi(y)|^p|u(y)|^p}{|x-y|^{n+sp}}\chi_{B_1(y)^c}(x)\,dxdy=\vcentcolon I_1+I_2,
    \end{split}
    \]
    where $\chi_A$ denotes the characteristic function of the set $A\subset \R^n$. Next we estimate $I_1,I_2$. We have
    \[
    \begin{split}
        I_1&\leq [\phi]_{C^{0,\mu}(\Omega)}^p\int_{\Omega}|u(y)|^p\left(\int_{B_1(y)}\frac{dx}{|x-y|^{n+(s-\mu)p}}\right)dy\\
        &=[\phi]_{C^{0,\mu}(\Omega)}^p\|u\|_{L^p(\Omega)}^p\int_{B_1(0)}\frac{dz}{|z|^{n+(s-\mu)p}}=\frac{\omega_n}{(\mu-s)p}[\phi]_{C^{0,\mu}(\Omega)}^p\|u\|_{L^p(\Omega)}^p
    \end{split}
    \]
    and
    \[
    \begin{split}
        I_2&\leq 2^p\|\phi\|_{L^{\infty}(\Omega)}^p\int_{\Omega}|u(y)|^p\left(\int_{B_1(y)^c}\frac{dx}{|x-y|^{n+sp}}\right)\,dy\\
        &=2^p\|\phi\|_{L^{\infty}(\Omega)}^p\|u\|_{L^p(\Omega)}^p\int_{B_1(0)^c}\frac{dz}{|z|^{n+sp}}=\frac{2^p\omega_n}{sp}\|\phi\|_{L^{\infty}(\Omega)}^p\|u\|_{L^p(\Omega)}^p,
    \end{split}
    \]
    where $\omega_n$ is the area of the unit sphere. Therefore, we get
    \[
    \begin{split}
        [\phi u]_{W^{s,p}(\Omega)}^p&\leq 2^{p-1}(1+\frac{2^p\omega_n}{sp}\|\phi\|_{L^{\infty}(\Omega)}^p+\frac{\omega_n}{(\mu-s)p}[\phi]_{C^{0,\mu}(\Omega)}^p)\|u\|_{W^{s,p}(\Omega)}^p\\
        &\leq C\left(1+\frac{\mu}{s(\mu-s)}\right)\|\phi\|_{C^{0,\mu}(\Omega)}^p\|u\|_{W^{s,p}(\Omega)}^p,
    \end{split}
    \]
    where $C$ only depends on $n$ and $p$. This establishes the assertion \ref{item: basic case}.\\
    
    Now let $s=k+\sigma$ with $k\in\N$ and $\sigma\in (0,1)$. By classical results we have $\phi u \in W^{k,p}(\Omega)$ with $\|\phi u\|_{W^{k,p}(\Omega)}\leq C\|\phi\|_{C^k(\Omega)}\|u\|_{W^{k,p}(\Omega)}$ for some $C>0$ only depending on $n,k$ and $p$. Thus it remains to estimate the Gagliardo seminorm of $\partial^{\alpha}(\phi u)$ for all multi-indices $\alpha$ of order $k$. By the Leibniz rule we have
    \[
    \begin{split}
        &[\partial^{\alpha}(\phi u)]_{W^{\sigma,p}(\Omega)}\leq C\sum_{\beta\leq \alpha}[\partial^{\alpha-\beta}\phi \partial^{\beta}u]_{W^{\sigma,p}(\Omega)}\\
        &\leq C\sum_{\beta\leq\alpha:\,\alpha\neq\beta}[\partial^{\alpha-\beta}\phi\partial^{\beta}u]_{W^{\sigma,p}(\Omega)}+C[\phi\partial^{\alpha}u]_{W^{\sigma,p}(\Omega)}\\
        &\leq C\sum_{\beta\leq \alpha:\,\alpha\neq \beta}\left(1+\frac{\mu}{\sigma(\mu-\sigma)}\right)^{1/p}\|\partial^{\beta-\alpha}\phi\|_{C^{0,\mu}(\Omega)}\|\partial^{\beta}u\|_{W^{\sigma,p}(\Omega)}\\
        &+C\left(1+\frac{\mu}{\sigma(\mu-\sigma)}\right)^{1/p}\|\phi\|_{C^{0,\mu}(\Omega)}\|\partial^{\alpha}u\|_{W^{\sigma,p}(\Omega)}\\
        &\leq CC_0\left(1+\frac{\mu}{\sigma(\mu-\sigma)}\right)^{1/p}\sum_{\beta\leq \alpha:\,\alpha\neq \beta}\|\partial^{\beta-\alpha}\phi\|_{C^{0,\mu}(\Omega)}\|\partial^{\beta}u\|_{W^{1,p}(\Omega)}\\
        &+C\left(1+\frac{\mu}{\sigma(\mu-\sigma)}\right)^{1/p}\|\phi\|_{C^{0,\mu}(\Omega)}\|\partial^{\alpha}u\|_{W^{\sigma,p}(\Omega)}\\
        &\leq C(1+C_0)\left(1+\frac{\mu}{\sigma(\mu-\sigma)}\right)^{1/p}\left(\sum_{\ell=0}^k\|\nabla^{\ell}\phi\|_{C^{0,\mu}(\Omega)}\right)\|u\|_{W^{s,p}(\Omega)}
    \end{split}
    \]
for all $\alpha\in \N^n_0$ with $|\alpha|=k$ and some constant $C>0$ only depending on $n,k$. In the last estimate we used the bound from the case $0<s<1$ and \cite[Proposition~2.2]{DINEPV-hitchhiker-sobolev}, where $C_0>0$ is the norm of the extension operator $E\colon W^{1,p}(\Omega)\to W^{1,p}(\R^n)$ (see \cite[Theorem~9.7]{FA-Brezis}).
\end{proof}

\begin{proof}[Proof of Lemma \ref{lem: zero extension}]
    First note that $\dist(\supp(u),\Omega^c)\geq d>0$. In fact, for $x\in\supp(u)\subset\Omega$, $y\in \Omega_e$ consider the curve $\gamma\colon [0,1]\to \R^n$ with $\gamma(t)\vcentcolon = x+t(y-x)$. Then there exists $t_0\in (0,1)$ such that $\gamma(t_0)\in \partial\Omega$ because otherwise one would have $[0,1]=\gamma^{-1}(\Omega)\cup \gamma^{-1}(\Omega_e)$ which is not possible. But then $|x-y|\geq |x-\gamma(t_0)|\geq d>0$. Next we distinguish the cases $0<s<1$, $s=k\in\N$ and $s=k+\sigma$ with $k\in\N$, $0<\sigma<1$.\\
    
    \noindent\textit{Case $0<s<1$}: Clearly, we have $\|\Bar{u}\|_{L^p(\R^n)}=\|u\|_{L^p(\Omega)}$ and thus it remains to show $[\Bar{u}]_{W^{s,p}(\R^n)}\leq C\|u\|_{W^{s,p}(\Omega)}$. By symmetry we can split $[\Bar{u}]_{W^{s,p}(\R^n)}$ as
    \[
    \begin{split}
        &\int_{\R^n}\int_{\R^n}\frac{|\Bar{u}(x)-\Bar{u}(y)|^p}{|x-y|^{n+sp}}\,dxdy=\int_{\Omega}\int_{\Omega}\frac{|u(x)-u(y)|^p}{|x-y|^{n+sp}}\,dxdy\\
        &+\int_{\Omega^c}\int_{\Omega^c}\frac{|\bar{u}(x)-\bar{u}(y)|^p}{|x-y|^{n+sp}}\,dxdy+2\int_{\Omega}\int_{\Omega^c}\frac{|u(x)|^p}{|x-y|^{n+sp}}\,dxdy\\
        &=[u]_{W^{s,p}(\Omega)}^p+2\int_{\supp(u)}|u(x)|^p\left(\int_{\Omega^c}\frac{dy}{|x-y|^{n+sp}}\,dy\right)dx.
    \end{split}
    \]
    For any $x\in \supp(u)$ there holds $B_{d/2}(x)\subset \Omega$ and hence we have
    \[
    \begin{split}
        \int_{\Omega^c}\frac{dy}{|x-y|^{n+sp}}\,dy&\leq  \int_{B_{d/2}(x)^c}\frac{dy}{|x-y|^{n+sp}}=\omega_n\int_{d/2}^{\infty}\frac{r^{n-1}}{r^{n+sp}}\,dr=\frac{\omega_n}{sp}\left(\frac{d}{2}\right)^{-sp}.
    \end{split}
    \]
    Therefore we get
    \[
        [\bar{u}]_{W^{s,p}(\R^n)}^p\leq [u]_{W^{s,p}(\Omega)}^p+2^{1+sp}\frac{\omega_n}{sp}d^{-sp}\|u\|_{L^p(\Omega)}^p.
    \]
    This shows
    \[
    \|\bar{u}\|_{W^{s,p}(\R^n)}\leq C\left(1+\frac{2^s}{(sp)^{1/p}}d^{-s}\right)\|u\|_{W^{s,p}(\Omega)}.
    \]
    
    \noindent\textit{Case $s=k\in\N$}: Let $\Omega_d\subset \Omega$ be the $d/2$-neighborhood of $\supp(u)$. Let $\phi\in C_c^{\infty}(\R^n)$. Then $\supp(\phi)\cap \overline{\Omega_d} = \emptyset$ or $\supp(\phi)\cap \overline{\Omega_d}\subset\Omega$ is a nonempty compact set. In the latter case choose a cutoff function $\eta_d\in C_c^{\infty}(\Omega)$, $0\leq \eta_d\leq 1$ with $\eta|_{V}=1$, where $V\Subset \Omega$ such that $\supp(\phi)\cap\overline{\Omega_d}\subset V$. Then for any $u\in W^{1,p}(\Omega)$ there holds
    \[
    \begin{split}
        \int_{\R^n}\bar{u}\partial_i\phi\,dx&=\int_{\Omega}u\eta_d\partial_i\phi\,dx=\int_{\Omega}u\partial_i(\eta_d\phi)\,dx-\int_{\Omega}u\phi\partial_i\eta_d\,dx\\
        &=-\int_{\Omega}\partial_iu(\eta_d\phi)\,dx=-\int_{\Omega}\partial_iu (\eta_d-1)\phi\,dx-\int_{\Omega}(\partial_iu)\phi\,dx\\
        &=-\int_{\Omega}(\partial_iu)\phi\,dx=-\int_{\R^n}\overline{\partial_iu}\phi\,dx
    \end{split}
    \]
    for all $1\leq i\leq n$. This identity clearly also holds if the intersection is empty. Thus $\partial_i\bar{u}=\overline{\partial_i u}\in L^p(\R^n)$ for all $1\leq i\leq n$. Hence, if $u\in W^{1,p}(\Omega)$ then $\bar{u}\in W^{1,p}(\R^n)$ and by the previous case there holds $\|\bar{u}\|_{W^{1,p}(\R^n)}= \|u\|_{W^{1,p}(\Omega)}$. By induction we see that for any $k\in\N$ we have $\bar{u}\in W^{k,p}(\R^n)$ whenever $u\in W^{k,p}(\Omega)$. Moreover, there holds $\|\bar{u}\|_{W^{k,p}(\R^n)}=\|u\|_{W^{k,p}(\Omega)}$.\\
    
    \noindent\textit{Case $s=k+\sigma$ with $k\in\N$, $0<\sigma<1$}: This follows immediately from the previous two cases.
\end{proof}

\bibliography{refs} 

\bibliographystyle{alpha}

\end{document}